\documentclass[a4paper, 10pt, twoside, notitlepage]{amsart}

\usepackage[utf8]{inputenc}
\usepackage{xcolor}
\usepackage{amsmath} 
\usepackage{amssymb} 
\usepackage{amsthm}
\usepackage{geometry}
\usepackage{graphicx}
\usepackage{mathtools}
\usepackage{esint}
\usepackage{enumitem}
\usepackage{comment}
\usepackage[colorlinks=true,linkcolor=blue]{hyperref}

\theoremstyle{plain}
\newtheorem{thm}{Theorem}
\newtheorem{prop}{Proposition}[section]
\newtheorem{lem}[prop]{Lemma}

\newtheorem{rmk}[prop]{Remark}

\newcommand {\R} {\mathbb{R}} 
 \newcommand {\N} {\mathbb{N}}
 
\newcommand {\p} {\partial}

\newcommand {\D} {\Delta}

\newcommand {\supp} {\text{supp}}

\DeclareMathOperator{\F} {\mathcal{F}}

\allowdisplaybreaks

\title[Nonlocal Local Reduction]{A Reduction of the Fractional Calder\'on Problem to the Local Calder\'on Problem by Means of the Caffarelli-Silvestre Extension}

\author[G. Covi]{Giovanni Covi}
\address{Institute for Applied Mathematics, University of Bonn, Endenicher Allee 60, 53115 Bonn, Germany}
\email{giovanni.covi@uni-bonn.de}

\author[T. Ghosh]{Tuhin Ghosh}
\address{Department of Mathematics, University of Bielefeld}
\email{tghosh@math.uni-bielefeld.de}

\author[A. R\"uland]{Angkana R\"uland}
\address{Institute for Applied Mathematics and Hausdorff Center for Mathematics, University of Bonn, Endenicher Allee 60, 53115 Bonn, Germany}
\email{rueland@uni-bonn.de}

\author[G. Uhlmann]{Gunther Uhlmann}
\address{Institute for Advanced Study, Hong Kong University of Science and
Technology, and Department of Mathematics, University of Washington}
\email{gunther@math.washington.edu}

\begin{document}

\begin{abstract}
We relate the (anisotropic) variable coefficient local and nonlocal Calder\'on problems by means of the Caffarelli-Silvestre extension. In particular, we prove that (partial) Dirichlet-to-Neumann data for the fractional Calder\'on problem in three and higher dimensions determine the (full) Dirichlet-to-Neumann data for the local Calder\'on problem. As a consequence, any (variable coefficient) uniqueness result for the local problem also implies a uniqueness result for the nonlocal problem. Moreover, our approach is constructive and associated Tikhonov regularization schemes can be used to recover the data. Finally, we highlight obstructions for reversing this procedure, which essentially consist of two one-dimensional averaging processes.
\end{abstract}

\maketitle
\tableofcontents

\section{Introduction}
\label{sec:intro}

It is the objective of this article to study the relation between the fractional and the classical Calder\'on problems. This had been initiated in the work \cite{GU21} where the solvability of the fractional Calder\'on problem (possibly with variable coefficient metric) is reduced to the solvability of the classical Calder\'on problem. In \cite{GU21} this is achieved by using the relation between the fractional Laplacian and the heat extension as, for instance, highlighted in \cite{ST10}. It is the main aim of this article to revisit the reduction of the fractional Calder\'on problem to the classical Calder\'on problem, adopting a slightly different perspective by relying on the Caffarelli-Silvestre extension \cite{CS07} instead of the heat extension.

Let us first take a heuristic, non-rigorous perspective on this and recall the classical and fractional Calder\'on problems. 

On the one hand, given a bounded, open, sufficiently regular domain $\Omega \subset \R^n$ in the classical Calder\'on problem one seeks to recover an unknown, possibly anisotropic conductivity $a: \Omega \rightarrow \R^{n\times n}_{sym}$ which is positive definite, symmetric and of a suitable regularity from voltage-to-current measurements on the boundary. In other words, one seeks to reconstruct the matrix $a$ given the measurements
\begin{align}
\label{eq:local_DN}
\Lambda: H^{\frac{1}{2}}(\partial \Omega) \rightarrow H^{-\frac{1}{2}}(\partial \Omega),\
f \mapsto \nu \cdot a \nabla u|_{\partial \Omega},
\end{align} 
where $u$ is a solution to the conductivity equation
\begin{align}
\label{eq:local_Cald}
\begin{split}
\nabla \cdot a \nabla u & = 0 \mbox{ in } \Omega,\\
u & = f \mbox{ on } \partial \Omega.
\end{split}
\end{align}
This problem has been studied intensively: {For the isotropic setting landmark results are available, including the seminal results on uniqueness \cite{SU87,HT13,CR16,H15}, stability \cite{A88} and recovery \cite{N88}. In contrast, the anisotropic problem is still widely open and has only been solved under strong structural conditions \cite{LU89,LU01,LTU03}. We refer to the survey article \cite{U09} for further references on the extensive literature on this.}

In the fractional Calder\'on problem, on the other hand, given a domain $\Omega \subset \R^n$ and a disjoint, sufficiently regular domain $W\subset \R^n$ and $s\in (0,1)$, one similarly seeks to recover an unknown conductivity $a: \R^n \rightarrow \R^{n\times n}_{sym}$ which is positive definite, symmetric and of a suitable regularity class, however from the fractional measurements
\begin{align}
\label{eq:DTN_nonlocal}
{\Lambda_s}: \tilde{H}^{s}(W) \rightarrow H^{-s}(W),
\ f \mapsto {(-\nabla \cdot a \nabla )^s u|_{W}},
\end{align}
where  $u$ is a solution to the fractional conductivity equation
\begin{align}
\label{eq:fracCald}
\begin{split}
(-\nabla \cdot a \nabla)^s u & = 0 \mbox{ in } \Omega,\\
u & = f \mbox{ on } \Omega_e:= \R^n \setminus \overline{\Omega}.
\end{split}
\end{align}
Here $(-\nabla \cdot a \nabla)^s $ can (equivalently) be understood spectrally, by means of an integral kernel, or also by means of a Caffarelli-Silvestre type extension \cite{CS07,ST10}.

In what follows, we will mainly adopt the Caffarelli-Silvestre extension perspective. More precisely, for $s\in (0,1)$ the Caffarelli-Silvestre extension formulation for the Dirichlet problem \eqref{eq:fracCald} reads
\begin{align}
\label{eq:CS}
\begin{split}
\nabla \cdot x_{n+1}^{1-2s} \tilde{a}(x') \nabla \tilde{u} & = 0 \mbox{ in } \R^{n+1}_+,\\
\lim\limits_{x_{n+1} \rightarrow 0} x_{n+1}^{1-2s} \p_{n+1} \tilde{u} & = 0 \mbox{ on } \Omega \times \{0\},\\
\tilde{u} & = f \mbox{ on } \Omega_e \times \{0\}. 
\end{split}
\end{align}

Here, as above, $\Omega, W \subset \R^n$ are disjoint open bounded sets, $f\in \tilde{H}^s(W)$ and the solution $\tilde{u} \in \dot{H}^{1}(\R^{n+1}_+,x_{n+1}^{1-2s})$. {The matrix-valued function $\tilde{a} \in C^2(\overline{\R^{n+1}_+},\R^{(n+1)\times (n+1)})$ is of the form 
\begin{align}
\label{eq:matrix}
\tilde{a}(x',x_{n+1}):=\tilde{a}(x'):= \begin{pmatrix} a(x') & 0 \\ 0 & 1 \end{pmatrix},
\end{align}
where $a\in C^2(\R^n,\R^{n\times n}_{sym})$ is assumed to be uniformly elliptic.} Solutions {to problem \eqref{eq:CS}} are understood in terms of the associated bilinear forms. Moreover, the fractional conductivity operator in \eqref{eq:DTN_nonlocal} can then be expressed in terms of a local Dirichlet-to-Neumann map \cite{CS07, ST10}:
\begin{align*}
(-\nabla \cdot a \nabla)^s u = c_s \lim\limits_{x_{n+1} \rightarrow 0} x_{n+1}^{1-2s} \p_{n+1} \tilde{u}(x',x_{n+1}) \in {\dot{H}^{-s}(\R^n)}, 
\end{align*} 
with both objects understood in their weak forms.
In particular, the fractional Calder\'on problem \eqref{eq:fracCald}, \eqref{eq:DTN_nonlocal} can be rephrased in terms of recovering the metric $\tilde{a}$ by means of the weighted Dirichlet-to-Neumann data
\begin{align*}
{\Lambda_s} : \tilde{H}^s(W) \rightarrow H^{-s}(W), \
f \mapsto {c_s \lim\limits_{x_{n+1} \rightarrow 0} x_{n+1}^{1-2s} \p_{n+1} \tilde{u}(x',x_{n+1})|_{W}}.
\end{align*}

\subsection{Main results}
Given the outlined perspective on the fractional Calder\'on problem, we seek to relate the fractional Calder\'on problem in its Caffarelli-Silvestre extension formulation
to the classical local Calder\'on problem given in \eqref{eq:local_Cald}, \eqref{eq:local_DN}. This is achieved by considering the following function 
\begin{align}
\label{eq:local}
\begin{split}
v(x') & = \int\limits_{0}^{\infty} x_{n+1}^{1-2s} \tilde{u}(x',x_{n+1}) dx_{n+1}, \mbox{ for } x' \in  \Omega,
\end{split}
\end{align}
which will be shown to be a solution to the classical Calder\'on problem with metric $a$, given the nonlocal data for $u$.
As above this is understood in its weak form. 

More precisely, as one of our main results, we obtain the following relation between the two Calder\'on type problems:

\begin{thm}
\label{thm:main}
{Let $n\in \N$, $n\geq 3$}.
Let $\Omega, W \subset \R^n$ be bounded Lipschitz sets with $\overline{\Omega}\cap \overline{W} = \emptyset$. Let $s\in (0,1)$, let {$\tilde{a} \in C^2(\overline{\R^{n+1}_+},\R^{(n+1)\times (n+1)})$} be of the form $\tilde{a}(x')= \begin{pmatrix} a(x') & 0 \\ 0 & 1 \end{pmatrix}$ where $a\in {C^2(\R^n,\R^{n\times n}_{sym})}$ is uniformly elliptic and such that $a\equiv Id$ in $\Omega_e$. Let $\tilde{u}$ be a weak solution to \eqref{eq:CS}. Then, we have that
\begin{align}
\label{eq:v}
v(x'):=\int\limits_{0}^{\infty}  x_{n+1}^{1-2s} \tilde{u}(x',x_{n+1}) dx_{n+1} \in H^1(\Omega),
\end{align}
and the function $v$ is a weak solution to \eqref{eq:local_Cald}. Moreover, the map 
\begin{align}
\label{eq:DNs}
\Lambda_s: \tilde{H}^s(W) \rightarrow H^{-s}(W), \ f \mapsto \lim\limits_{x_{n+1}\rightarrow 0} x_{n+1}^{1-2s} \p_{n+1}\tilde{u}|_{W }
\end{align}
determines the map
\begin{align}
\label{eq:loc}
\Lambda: H^{\frac{1}{2}}(\partial \Omega) \rightarrow H^{- \frac{1}{2}}(\partial \Omega), \ g \mapsto \nu \cdot \tilde{a} \nabla v|_{\partial \Omega}.
\end{align}
\end{thm}

As in \cite{GU21} a direct consequence of Theorem \ref{thm:main} is the reduction of the nonlocal Calder\'on problem to the local one. Hence, uniqueness results for the local Calder\'on problem, e.g. for analytic manifolds or for the reconstruction of conformal factors in the presence of limiting Carleman weights, directly imply solvability of the nonlocal problem.

\begin{thm}
\label{thm:consequences}
{Let $n\in \N$, $n\geq 3$}.
Let $\Omega, W \subset \R^n$ be bounded Lipschitz sets with $\overline{\Omega}\cap \overline{W} = \emptyset$. Let $s\in (0,1)$, and 
let {$\tilde{a} \in C^2(\overline{\R^{n+1}_+},\R^{(n+1)\times (n+1)})$} be of the form $\tilde{a}(x')= \begin{pmatrix} a(x') & 0 \\ 0 & 1 \end{pmatrix}$ where $a\in {C^2(\R^n,\R^{n\times n}_{sym})}$ is uniformly elliptic and such that $a\equiv Id$ in $\Omega_e$. Assume that the local Dirichlet-to-Neumann map $\Lambda:H^{\frac{1}{2}}(\partial \Omega) \rightarrow H^{- \frac{1}{2}}(\partial \Omega)$ uniquely determines the coefficients $a$ in the classical Calder\'on problem \eqref{eq:local_Cald}. Then the nonlocal Dirichlet-to-Neumann  map  $\Lambda_s: \tilde{H}^s(W) \rightarrow H^{-s}(W)$ uniquely determines the coefficients $a$ in the fractional Calder\'on problem \eqref{eq:fracCald}.
\end{thm}

We further stress that it is also possible to apply analogous arguments to fractional Calder\'on type problems in which the operator 
\begin{align*}
(-\nabla \cdot a \nabla)^s , \qquad s\in (0,1),
\end{align*}
is replaced by operators of the form
\begin{align*}
(-\nabla \cdot a \nabla - ( i b \cdot \nabla + i \nabla \cdot b) + c)^s 
\end{align*}
for $b: \R^n \rightarrow \R^n$, $c: \R^n \rightarrow \R$ sufficiently regular. {In particular, uniqueness results for associated local Calder\'on type problems directly imply uniqueness results for the nonlocal analogue.}\\

Moreover, the determination of the local data from the nonlocal ones is completely constructive. Indeed, the functions $(\tilde{u}, \p_{\nu} \tilde{u})$ can be obtained on $\partial \Omega \times \R_+$ by unique continuation from the data $(f,\Lambda_s f)$. This then yields the functions $(v,\p_{\nu}v)$ by integration. The unique continuation result can, for instance, be formulated by means of the Tikhonov regularization procedure.

\begin{prop}
\label{prop:Tikhonov}
{Let $n\in \N$, $n\geq 3$}.
Let $\Omega, W \subset \R^n$ be bounded Lipschitz sets with $\overline{\Omega}\cap \overline{W} = \emptyset$. Let $s\in (0,1)$, $\varepsilon > 0$, and assume that {$\tilde{a} \in C^2(\overline{\R^{n+1}_+},\R^{(n+1)\times (n+1)})$} is of the form $\tilde{a}(x')= \begin{pmatrix} a(x') & 0 \\ 0 & 1 \end{pmatrix}$ where $a\in {C^2(\R^n,\R^{n\times n}_{sym})}$ is uniformly elliptic and such that $a\equiv Id$ in $\Omega_e$. Define the operator 
\begin{align*}
A: \ &\mathcal{V} \rightarrow \tilde{H}^{s-\epsilon}(W) \times H^{-s-\epsilon}(W)  ,\\
&\tilde{u} \mapsto (\tilde u(\cdot,0), \lim\limits_{x_{n+1} \rightarrow 0} x_{n+1}^{1-2s}\p_{n+1} \tilde{u}),
\end{align*}
where $$\mathcal{V}:=\{\tilde{u} \in \dot{H}^{1}(\Omega_e \times \R_+, x_{n+1}^{1-2s}):\;  \nabla \cdot x_{n+1}^{1-2s} \nabla \tilde{u} = 0 \mbox{ in } \Omega_e \times \R_+, \;  \supp(\tilde{u}(\cdot,0)) \subset \overline{W}\}.$$
Finally, let $(f,\Lambda_sf)\in \tilde{H}^{s}(W) \times H^{-s}(W)$. Then for each $\alpha > 0$ there exists a unique minimizer $\tilde u_\alpha \in \mathcal V$ of the functional
\begin{align*}
J_{\alpha}(\tilde{u}) 
:&= \|A \tilde u - (f,\Lambda_sf) \|_{H^{s-\epsilon}(W)\times H^{-s-\epsilon}(W)}^2 +  \alpha\|x_{n+1}^{\frac{1-2s}{2}} \nabla \tilde{u}\|_{L^2(\Omega_e \times \R_+)}^2
\\ & = \|\tilde{u}(\cdot,0)-f\|_{H^{s-\epsilon}(W)}^2 
+ \|\lim\limits_{x_{n+1} \rightarrow 0} x_{n+1}^{1-2s}\p_{n+1} \tilde{u}- \Lambda_sf\|_{H^{-s-\epsilon}(W)}^2 +  \alpha\|x_{n+1}^{\frac{1-2s}{2}} \nabla \tilde{u}\|_{L^2(\Omega_e \times \R_+)}^2.
\end{align*}
In particular, for each $(f,\Lambda_s f)\in \tilde{H}^{s}(W) \times H^{-s}(W)$ it holds that $A \tilde{u}_{\alpha} \rightarrow (f,\Lambda_s f)$ in $\tilde{H}^{s-\epsilon}(W) \times H^{-s-\epsilon}(W)$.  
\end{prop}

We expect that the described technique can also be made quantitative by invoking quantitative unique continuation results. We plan to address this in future work.

\subsection{Obstructions on the reversal of the outlined procedure}
The outlined procedure provides a clear relation between the local and fractional Calder\'on problems of the form described above. It allows to deduce uniqueness results for the \emph{fractional} Calder\'on problem from possibly available uniqueness results for the \emph{local} Calder\'on problem. A natural question deals with the reversibility of this procedure. Due to the different information content of the local versus the fractional Calder\'on problems -- in that the Dirichlet-to-Neumann operator for the local problem formally determines $2n-2$ degrees of freedom while the nonlocal one determines $2n$ degrees of freedom -- it is not expected that this can be fully reversed. Our approach accounts precisely for this mismatch of formally determined degrees of freedom: In carrying out the integration procedure which allows to pass from the function $\tilde{u}$ to the function $v$ and by restricting to the boundary of $\partial \Omega$ we lose exactly the given degrees of freedom that distinguish the local and fractional data. 
We reformulate this precisely:

\begin{prop}
\label{prop:density_inversion}
{Let $n\in \N$, $n\geq 3$ and let $s\in (0,1)$}.
Let $\Omega, W \subset \R^n$ be bounded Lipschitz sets with $\overline{\Omega}\cap \overline{W} = \emptyset$. Let $s\in (0,1)$, let {$\tilde{a} \in C^2(\overline{\R^{n+1}_+},\R^{(n+1)\times (n+1)})$} be of the form $\tilde{a}(x')= \begin{pmatrix} a(x') & 0 \\ 0 & 1 \end{pmatrix}$ where $a\in {C^2(\R^n,\R^{n\times n}_{sym})}$ is uniformly elliptic and such that $a\equiv Id$ in $\Omega_e$.
Let
\begin{align*}
\mathcal{C}_{s,a} &:= \left\{(f, \Lambda_s(f)): \ f\in \tilde{H}^{s}(W) \right\} \subset \tilde{H}^{s}(W) \times H^{-s}(W),\\
\mathcal{C}_{a}&:= \left\{(f, \Lambda(f)): \ f\in {H}^{\frac{1}{2}}(\partial \Omega) \right\} \subset H^{\frac{1}{2}}(\partial \Omega) \times H^{-\frac{1}{2}}(\partial \Omega).
\end{align*}
Then, there exists a bounded linear operator $T: \mathcal{C}_{s,a} \rightarrow \mathcal{C}_{a}$ such that
\begin{align*}
\overline{T(\mathcal{C}_{s,a})}^{H^{\frac{1}{2}}(\partial \Omega) \times H^{-\frac{1}{2}}(\partial \Omega)} = \mathcal{C}_{a} .
\end{align*}
Here the operator $T:\mathcal{C}_{s,a} \rightarrow \mathcal{C}_s$ is given by
\begin{align*}
T(f,\Lambda_s(f)) := (v|_{\partial \Omega}, \p_{\nu} v|_{\partial \Omega}),
\end{align*}
with $\nu$ denoting the outer normal to $\partial \Omega$ and $v(x'):= \int\limits_{0}^{\infty} x_{n+1}^{1-2s} \tilde{u}_f(x',x_{n+1}) dx_{n+1}$; the function $\tilde{u}_f$ being the Caffarelli-Silvestre extension associated with the solution $u_f$ of \eqref{eq:fracCald} with data $f$. 

\end{prop}

The obstructions for reversing the procedure outlined in Theorem \ref{thm:main} thus exactly stem from inverting the bounded linear operator $T$ and the density result in Proposition \ref{prop:density_inversion} (instead of an onto mapping property).
This obstruction can also be factored into two steps: The first consists of the obstruction of recovering
\begin{align*}
(x_{n+1}^{1-2s}\tilde{u}|_{\partial \Omega \times \R_+}, x_{n+1}^{1-2s} \partial_{\nu} \tilde{u}|_{\partial \Omega \times \R_+})
\end{align*}
from the averaged data
\begin{align*}
(v|_{\partial \Omega }, \partial_{\nu}  v|_{\partial \Omega} ).
\end{align*}
The second obstruction is given by continuing the data $(x_{n+1}^{1-2s}\tilde{u}|_{\partial \Omega \times \R_+}, x_{n+1}^{1-2s} \partial_{\nu} \tilde{u}|_{\partial \Omega \times \R_+})$ to the data $(f,\Lambda_s(f))$. 
The crucial difficulty here consists in the first obstruction. Indeed, while the second continuation step is certainly highly (i.e., logarithmically) unstable, it consists of a unique continuation argument (in an unbounded domain), which is comparably well-understood.

It is an interesting question to study the inversion properties of the operator $T$ and the associated obstructions to inverting the outlined reduction argument in more detail.  We plan to explore this further in future work.

\subsection{Outline of the proof}
In concluding the introduction, let us give the formal, non-rigorous argument for the relation between $\tilde{u}$ and $v$:
Integrating the equation $\nabla \cdot x_{n+1}^{1-2s} a \nabla \tilde{u} = 0$ with respect to the $x_{n+1}$ variable yields
\begin{align*}
0= \nabla' \cdot a \nabla' v + \int\limits_{0}^{\infty} \p_{n+1}x_{n+1}^{1-2s} \p_{n+1} \tilde{u}(x', x_{n+1}) dx_{n+1} = \nabla' \cdot a \nabla' v,
\end{align*}
for $x' \in \Omega$, where we have used that $\int\limits_{0}^{\infty} \p_{n+1}x_{n+1}^{1-2s} \p_{n+1} \tilde{u}(x', x_{n+1}) dx_{n+1} =0$ holds by the boundary condition on $\Omega \times \{0\}$ and the decay of $\tilde{u}$ (and its derivatives) at infinity.

Moreover, formally, the associated Dirichlet-to-Neumann map is given by
\begin{align*}
v(x'):=\int\limits_{0}^{\infty}  x_{n+1}^{1-2s} \tilde{u}(x',x_{n+1}) dx_{n+1} \mapsto \p_{\nu} v(x') :=\p_{\nu} \int\limits_{0}^{\infty} x_{n+1}^{1-2s} \tilde{u}(x',x_{n+1}) dx_{n+1} \mbox{ for } x' \in \partial \Omega.
\end{align*}

A key step in our argument for Theorem \ref{thm:main} will be the rigorous derivation of these equations together with the proof of the fact that the knowledge of the (partial) nonlocal Dirichlet-to-Neumann map \eqref{eq:DTN_nonlocal} yields the (up to closure) full Dirichlet-to-Neumann map for the nonlocal problem. In this context we prove that the set of functions $v(x') \in H^{\frac{1}{2}}(\partial \Omega)$ which are obtained by this procedure form a dense subset of $H^{\frac{1}{2}}(\partial \Omega)$. Hence, the full Dirichlet-to-Neumann map for the local problem is determined from the nonlocal one. In particular, if there is uniqueness on the level of the local problem for the matrix $a$, then one also has uniqueness of the coefficients $a$ in the nonlocal problem. A key observation which will be relevant for our arguments consists of a duality principle which we will discuss in the next section.

In the remainder of the article, we make the above outlined arguments precise, and provide the density result and the argument for the constructive Tikhonov approximation of the Dirichlet-to-Neumann data.

\subsection{Relation to the literature}
Since the introduction and the derivation of partial data uniqueness results for the fractional Calder\'on problem in the seminal work \cite{GSU20}, there has been an intensive study of nonlocal inverse problems of Calder\'on type. This includes the investigation of low regularity partial data uniqueness \cite{RS20}, stability and instability results \cite{RS20,RS18} and (single measurement) recovery results for potentials \cite{GRSU20, R21,HL19,HL20}, the analysis of Liouville type transforms as well as the study of lower order drift type contributions \cite{C20,C20b,CLR20,CMRU20,BGU21}. Further there is a, by now, broad literature on related nonlocal models and the associated techniques have also proved of relevance in local problems, see for instance \cite{RS19,GFRZ22,RS19a,R19,L20,LO20,CDHS22,LLR20,L21} and the references therein for a non-exhaustive list on the by now large field. Many of these results rely on the strong global unique continuation properties of the fractional Laplacian \cite{FF14, R15}. We refer to the surveys \cite{S17, R18a} as well as to the literature in the above cited articles for further references.

Moreover, more recently, also the recovery of leading order contributions has been investigated in \cite{F21, FGKU21} from source type data. Here a key ingredient consists of the heat kernel representation of the fractional Laplacian and Kannai type transforms. These allow to transfer the problem to a hyperbolic problem which can be addressed by means of the boundary control method.

Due to the very strong results for the fractional Calder\'on problem -- of which many local counterparts are open problems -- the relation between local and nonlocal Calder\'on type problems is a major question. In this context, in \cite{CR21} a relation between the local Robin problem and fractional Calder\'on problems involving lower order potentials had been outlined. In \cite{GU21} a first result on the relation of Calder\'on type problems as in the present article was established building on heat kernel representations of the fractional Laplacian. In this article, we complement this with a (degenerate) elliptic approach on the relation between these two problems. This in particular allows for completely constructive results and also to identify key obstructions on reversing this procedure (formulated in Proposition \ref{prop:density_inversion} above). We plan to study this further in future work.

\subsection{Notation}\label{sec:not}
Before turning to the main body of the text, we introduce some notation which will be frequently used in the remainder of the article:
\begin{itemize}
\item Whenever an inequality (respectively, an identity) holds up to a positive constant whose exact value is not relevant for our arguments, we will use the symbols $\lesssim$ and $ \gtrsim$ (respectively $\approx$).
\item We set $\R^{n+1}_+:=\{(x',x_{n+1})\in \R^{n+1}: \ x_{n+1}>0\}$ and use the notation $\overline{\R^{n+1}_+}:= \{(x',x_{n+1})\in \R^{n+1}: \ x_{n+1}\geq 0\}$.
\item In order to denote the symmetric $n\times n$ matrices we write $\R^{n\times n}_{sym}$. In case that we seek to denote positive definiteness, we may also use the index `$+$'.
\item We also define the local operator $L:= \nabla'\cdot a\nabla'$. Here $\nabla'$ indicates the gradient in $\R^n$ (i.e., in the \emph{tangential directions}), as opposed to the symbol $\nabla$, which for us indicates differentiation in $\R^{n+1}_+$.
\item  Next, we define a number of function spaces on which our arguments will rely. Following \cite[Section 3]{McLean}, for $s\in \R$ we define the inhomogeneous fractional Sobolev space
$$ H^s(\R^n):= \{ u\in \mathcal S'(\R^n) : \|u\|_{H^s(\R^n)} := \| (1+|\xi|^2)^{s/2}\hat u(\xi)\|_{L^2(\R^n)} <\infty \}, $$ and the homogeneous fractional Sobolev space $$  \dot H^s(\R^n):= \{ u\in \mathcal S'(\R^n) : \|u\|_{\dot H^s(\R^n)} := \| |\xi|^s\hat u(\xi)\|_{L^2(\R^n)} <\infty \}. $$ If $\Omega\subset\R^n$ is open, bounded and has Lipschitz boundary, we also let $$ H^s(\Omega):=\{u|_\Omega, u\in H^s(\R^n)\}, $$ equipped with the quotient norm $\|u\|_{H^s(\Omega)} := \inf\{ \|U\|_{H^s(\R^n)} : U|_\Omega=u \}. $ Moreover, we define
$$\widetilde H^s(\Omega):= \overline{C^\infty_c(\Omega)}^{H^s(\R^n)}, \qquad H^s_{\overline\Omega} := \{ u\in H^s(\R^n) : \supp( u)\subseteq \overline\Omega\},$$ and observe that $\widetilde H^s(\Omega) = H^s_{\overline\Omega}$ holds under our regularity assumptions (which will always presuppose that $\Omega$ is Lipschitz). Further, we have $$ (\widetilde H^s(\Omega))^* = H^{-s}(\Omega), \qquad\mbox{and}\qquad (H^s(\Omega))^* = \widetilde H^{-s}(\Omega).$$ The corresponding fractional Sobolev spaces of the homogeneous kind on a bounded set can be defined in an analogous way. We will also make use of the following weighted (inhomogeneous and homogeneous) Sobolev spaces, which we define for $s\in (0,1)$:
\begin{align*}
 H^1(\R^{n+1}_+,x_{n+1}^{1-2s}) &:= \{ u:\R^{n+1}_+\rightarrow \R  : \| x_{n+1}^{\frac{1-2s}{2}}u \|_{L^2(\R^n)} + \| x_{n+1}^{\frac{1-2s}{2}}\nabla u \|_{L^2(\R^n)} <\infty  \}, \\
 \dot H^1(\R^{n+1}_+,x_{n+1}^{1-2s}) &:= \{ u:\R^{n+1}_+\rightarrow \R : \| x_{n+1}^{\frac{1-2s}{2}}\nabla u \|_{L^2(\R^n)} <\infty  \}. 
 \end{align*}
\end{itemize}

\subsection{Organization of the article}
The remainder of the article is organized as follows: In Section \ref{sec:duality} we recall a duality principle which underlies our reduction of the fractional Calder\'on problem to its local 
counterpart. Here our key result is the converse of the usually used form of the duality which is encoded in Theorems \ref{prop:equation_1} and \ref{prop:equation_2}. In this context, we presuppose the well-definedness of the function $v$ from \eqref{eq:v}. This requires a discussion of regularity and integrability properties of the Caffarelli-Silvestre extension. Due to its technical character, this is postponed to Section \ref{sec:reg} below.
Building on the results of Theorems \ref{prop:equation_1}, \ref{prop:equation_2}, in Section \ref{sec:density-2} we show that the local Dirichlet-to-Neumann map is determined by its nonlocal version. Here, due to the different dimensionalities of the Caffarelli-Silvestre extension (acting on $\R^{n+1}_+$) certain regularity and integrability questions are of relevance. Given the density result, in Sections \ref{sec:cons} and \ref{sec:tikh} we present the proofs of Theorem \ref{thm:consequences} and of Proposition \ref{prop:Tikhonov}. Finally, in Section \ref{sec:reg} we discuss the well-definedness and regularity of the function $v$ which is necessary for the above discussion.

\section{A duality principle}
\label{sec:duality}

In order to deduce the relation between the functions $v$ and $\tilde{u}$, we recall a general duality principle which was first formulated in \cite[Section 2.3]{CS07} (see also \cite[Proposition 3.6]{CS14}) for the constant coefficient setting. We state this for classical solutions to the variable coefficient problem. The central conclusion for weak solutions will be obtained in combination with the regularity results from Section \ref{sec:reg} (see Theorems  \ref{prop:equation_1} and \ref{prop:equation_2} below).

\begin{prop}[Duality]\label{prop:duality}
Let $s\in (0,1)$, $n\in \N$, let {$\tilde{a} \in C^2(\overline{\R^{n+1}_+},\R^{(n+1)\times (n+1)})$} be of the form $\tilde{a}(x')= \begin{pmatrix} a(x') & 0 \\ 0 & 1 \end{pmatrix}$ where $a\in {C^2(\R^n,\R^{n\times n}_{sym})}$ is uniformly elliptic. Let $h \in C^0(\R^n)$. Assume that $u_1 \in C^2(\R^{n+1}_+)$ with $x_{n+1}^{2s-1} \p_{n+1} u_1 \in C^0(\overline{\R^{n+1}_+})$ is a classical solution to 
\begin{equation}\begin{split}\label{eq:CS-Neumann}
\nabla \cdot x_{n+1}^{2s-1} \tilde  a \nabla u_1 &= 0 \mbox{ in } \R^{n+1}_+,\\
\lim\limits_{x_{n+1}\rightarrow 0} x_{n+1}^{2s-1} \p_{n+1} u_1 & = h \mbox{ on } \R^n \times \{0\}.
\end{split}
\end{equation}
Then the function $u_2(x',x_{n+1}):= x_{n+1}^{2s-1} \p_{n+1}u_1(x',x_{n+1})$ is a classical solution to 
\begin{equation}\begin{split}\label{eq:CS-Dirichlet}
\nabla \cdot  x_{n+1}^{1-2s} \tilde a \nabla u_2 &= 0 \mbox{ in } \R^{n+1}_+,\\
u_2 & = h \mbox{ on } \R^n \times \{0\}.
\end{split}\end{equation}
\end{prop}

For completeness, we include the short proof for this result.

\begin{proof}
We observe that for $x_{n+1}>0$ we have
\begin{align*}
\nabla \cdot x_{n+1}^{1-2s} \nabla u_2
&= \nabla \cdot x_{n+1}^{1-2s} \nabla x_{n+1}^{2s-1} \p_{n+1} u_1\\
&= \p_{n+1} \nabla' \cdot a\nabla' u_1 + \p_{n+1} x_{n+1}^{1-2s}(\p_{n+1} x_{n+1}^{{2s-1}} \p_{n+1} u_1)\\
&= \p_{n+1} x_{n+1}^{1-2s}(\nabla \cdot x_{n+1}^{{2s-1}} \tilde a \nabla  u_1) = 0.
\end{align*}
The identity at the boundary follows by definition and the regularity assumptions.
\end{proof}

In what follows we will be mainly interested in (suitable weak formulations of) the converse of the relation outlined in Proposition \ref{prop:duality}. 
Motivated by the discussion in the previous proposition, we therefore consider the function
\begin{align}
\label{eq:w}
w(x', x_{n+1}) := \int\limits_{x_{n+1}}^{\infty} t^{1-2s} \tilde{u}(x',t) dt, \qquad x_{n+1}>0,
\end{align}
where $\tilde{u} \in \dot{H}^{1}(\R^{n+1}_+, x_{n+1}^{1-2s})$ is a solution to the equation \eqref{eq:CS-Dirichlet} with $h \in \tilde{H}^s(W)$ and $W\subset \R^n$ an open, bounded Lipschitz domain. Using kernel estimates, it can be shown that the integral defining $w$ is finite for each $x_{n+1}>0$ and that for every $x_{n+1}>0$ the function $w$ is as regular as the coefficient $a$ permits (e.g., $C^{\infty}$ if $a$ is smooth). Moreover, the more detailed regularity estimates in Section \ref{sec:reg} even allow to conclude that the limit {$v(x'):=w(x',0)$} is well-defined and that ${v(x') }\in H^1(\Omega)$. We refer to Proposition \ref{prop:reg} in Section \ref{sec:reg} for an analysis of this function.

Now, the relevance of the function $w$ stems from the fact that for classical solutions $\tilde{u}$ with sufficient tangential and normal decay, it indeed reverses the relation formulated in Proposition \ref{prop:duality}. More precisely, for classical solutions one directly obtains from formal calculations that
\begin{align}
\label{eq:representation_s}
(-\nabla' \cdot a \nabla')^{1-s} v = -\tilde u(x',0) = -u(x') \mbox{ for } x' \in \R^n.
\end{align}
Here $(-\nabla' \cdot a \nabla')^{1-s} v$ is defined in the sense of Caffarelli-Silvestre extensions, that is
$$(-\nabla' \cdot a \nabla')^{1-s} v := c_{1-s}\lim\limits_{x_{n+1}\rightarrow 0} x_{n+1}^{2s-1} \p_{n+1}{\tilde{g}},$$
where $\tilde{g}$ is the classical, decaying solution of the Caffarelli-Silvestre extension problem 
\begin{align}\begin{split}
\label{eq:g}
\nabla \cdot  x_{n+1}^{2s-1} \tilde a \nabla {\tilde{g}} &= 0 \mbox{ in } \R^{n+1}_+,\\
{\tilde{g}} & = v \mbox{ on } \R^n.
\end{split}\end{align}
We note that by a short formal calculation, the function $w$ from \eqref{eq:w} (see the argument below) satisfies this equation.
Using the semi-group property of the fractional Laplacian, we (formally) immediately obtain that
\begin{align*}
(-\nabla' \cdot a \nabla') v = (-\nabla' \cdot a \nabla')^s u \mbox{ in } \R^n,
\end{align*}
and thus, in particular, $(-\nabla' \cdot a \nabla') v = 0$ in $\Omega$. 

Indeed, let us first give a non-rigorous argument for these relations: To this end, let us first show that $w$ formally satisfies \eqref{eq:g} (the rigorous derivation will be carried out in Theorem \ref{prop:equation_1} below). Here it suffices to prove the bulk equation, since the boundary equation is formally valid by construction. For the bulk equation, we observe that
\begin{align*}
\nabla \cdot x_{n+1}^{2s-1} \nabla w
&= \int\limits_{x_{n+1}}^{\infty} t^{1-2s} x_{n+1}^{2s-1} \Delta' \tilde{u}(x',t) dt - \p_{x_{n+1}} \tilde{u}(x', x_{n+1})\\
& = -  \int\limits_{x_{n+1}}^{\infty}  x_{n+1}^{2s-1} \p_{t} t^{1-2s} \p_t \tilde{u}(x',t) dt - \p_{x_{n+1}} \tilde{u}(x', x_{n+1})\\
& =  \p_{x_{n+1}} \tilde{u}(x', x_{n+1})- \p_{x_{n+1}} \tilde{u}(x', x_{n+1}) =0.
\end{align*}
Finally, formally, the identity \eqref{eq:representation_s} follows by definition of $w$:
\begin{align*}
\lim\limits_{x_{n+1} \rightarrow 0} x_{n+1}^{2s-1} \p_{x_{n+1}} w(x', x_{n+1})
= \lim\limits_{x_{n+1} \rightarrow 0} \tilde{u}(x',x_{n+1}) = u(x').
\end{align*}

With this discussion in hand, a key ingredient in our argument will thus rely on making these formal computations rigorous, which is the content of the next results and of the regularity estimates in Section \ref{sec:reg}.

\begin{thm}
\label{prop:equation_1}
{Let $s\in (0,1)$, $n \in \N$}.
{Let $\tilde{a} \in C^2(\overline{\R^{n+1}_+},\R^{(n+1)\times (n+1)})$} be of the form $\tilde{a}(x')= \begin{pmatrix} a(x') & 0 \\ 0 & 1 \end{pmatrix}$ where $a\in {C^2(\R^n,\R^{n\times n}_{sym})}$ is uniformly elliptic.
Let $\tilde u \in H^{1}(\R^{n+1}_+, x_{n+1}^{1-2s})$ be the Caffarelli-Silvestre extension of $u\in H^s(\R^n)$ with $\supp(u) \subset \R^n$ compact. Consider $v(x'):= \int_0^\infty t^{1-2s}\tilde u(x',t)dt$. 
Assume that $v \in H^1(\Omega')$ for some $\Omega' \subset \R^n$ open, bounded, Lipschitz.
Then, $v$ is a weak solution to 
\begin{align*}
\nabla' \cdot a \nabla' v = (-\nabla' \cdot a \nabla')^s u \mbox{ in } \Omega',
\end{align*}
i.e., the following identity holds:
\begin{align*}
\int\limits_{\Omega'} a(x') \nabla' v(x') \cdot \nabla' \varphi(x') dx' = \int\limits_{\R^n} \varphi(x') \lim\limits_{t\rightarrow 0} t^{1-2s} \p_t \tilde{u}(x',t) dx', \ \mbox{ for all } \varphi \in H^1_0(\Omega').
\end{align*}
\end{thm}

\begin{rmk}
We note that the expression 
\begin{align*}
\int\limits_{\R^n} \varphi(x') \lim\limits_{t\rightarrow 0} t^{1-2s} \p_t \tilde{u}(x',t) dx'
\end{align*}
is understood as an $H^{-s}-H^s$ duality pairing, which is well-defined since $\lim\limits_{t\rightarrow 0} t^{1-2s} \p_t \tilde{u}(x',t) \in \dot{H}^{-s}(\R^n)$ for solutions to the Caffarelli-Silvestre extension of $H^s(\R^n)$ functions with compact support.
\end{rmk}

\begin{proof}
We first note that, by density, it suffices to consider $\varphi \in C_c^{\infty}(\Omega')$.
Next, by construction of $v$, we obtain
\begin{align}
\label{eq:local_eq}
\begin{split}
\int\limits_{\Omega'} a \nabla' v \cdot \nabla' \varphi dx'
& = \int\limits_{\R^n} a\nabla' v \cdot \nabla' \varphi dx'\\
& = \int\limits_{\R^n} a\nabla' \int\limits_{0}^{\infty} t^{1-2s} \tilde{u}(x',t) dt \cdot \nabla' \varphi(x') dx'\\
& = \lim\limits_{k\rightarrow \infty} \int\limits_{\R^n} a \nabla' \int\limits_{0}^{\infty} t^{1-2s}\tilde{u}(x',t) \eta_k(t) dt \cdot \nabla' \varphi(x') dx'.
\end{split}
\end{align}
Here $\eta_k(t) = \eta_1(t/k)$, where $\eta_1 : [0,\infty) \rightarrow \R$ is a smooth function with $\eta_1(t) = 1$ for $t\in [0,1)$ and $\eta_1(t) = 0$ for $t\geq 2$. The convergence in \eqref{eq:local_eq} follows from the fact that 
\begin{align*}
\nabla' \int\limits_{0}^{\infty} t^{1-2s} \tilde{u}(x',t) \eta_k(t) dt \rightarrow \nabla' \int\limits_{0}^{\infty} t^{1-2s} \tilde{u}(x',t)  dt \mbox{ in } L^2(\Omega),
\end{align*}
since, by virtue of equation \eqref{eq:extension_est} in Lemma \ref{CS-estimates-vertical} we have that
\begin{align*}
\left\| \nabla' \int\limits_{0}^{\infty} t^{1-2s} \tilde{u}(x',t)(1- \eta_k)(t) dt \right\|_{L^2(\Omega)}
&\leq \left\| \int\limits_{k}^{2k} t^{1-2s}  |\nabla'\tilde{u}(x',t)| dt \right\|_{L^2(\Omega)}\\
& \lesssim k^{-n-2s} \|\tilde{u}(\cdot, 0)\|_{L^1(\Omega)}  \\ & \lesssim k^{-n-2s} \|\tilde{u}(\cdot, 0)\|_{L^2(\Omega)}.
\end{align*}

By the regularity of $\tilde{u}$, we compute that  
\begin{align*}
\int\limits_{\R^n} a\nabla' \int\limits_{0}^{\infty} t^{1-2s}\tilde{u}(x',t) \eta_k(t) dt \cdot \nabla' \varphi(x') dx' = \int\limits_{0}^{\infty} \int\limits_{\R^n} a t^{1-2s} \nabla'  \tilde{u}(x',t)  \cdot \nabla' (\varphi(x')\eta_k(t)) dx' dt.
\end{align*}
This equality can be obtained by expressing the tangential derivative as a limit of difference quotients, and then using dominated convergence and Fubini's theorem. 

Invoking the equation for $\tilde{u}$ then implies that
\begin{align*}
&\lim\limits_{k\rightarrow \infty} \int\limits_{0}^{\infty} \int\limits_{\R^n} t^{1-2s} a \nabla'  \tilde{u}(x',t)  \cdot \nabla' (\varphi(x')\eta_k(t)) dx' dt\\
& = - \lim\limits_{k\rightarrow \infty} \int\limits_{0}^{\infty} \int\limits_{\R^n} t^{1-2s}  \p_t  \tilde{u}(x',t)  \p_t (\varphi(x')\eta_k(t)) dx' dt
+ \int\limits_{\R^n} \lim\limits_{t\rightarrow 0} t^{1-2s} \p_t  \tilde{u}(x',t) \varphi(x') dx'.
\end{align*}
It remains to argue that 
\begin{align*}
\lim\limits_{k\rightarrow \infty} \int\limits_{0}^{\infty} \int\limits_{\R^n} t^{1-2s}  \p_t  \tilde{u}(x',t)  \p_t (\varphi(x')\eta_k(t)) dx' dt=0.
\end{align*}
Indeed, we have
\begin{align*}
 \lim\limits_{k\rightarrow \infty} \int\limits_{0}^{\infty} \int\limits_{\R^n} t^{1-2s}  \p_t  \tilde{u}(x',t)  \varphi(x') \p_t \eta_k(t) dx' dt=0,
\end{align*}
which is a consequence of the fact that $|\p_t  \tilde{u}(x',t)| \lesssim k^{-n-1}$ and $| \p_t \eta_k(t)|\leq C k^{-1}$ for $t\in (k,2k)$ (see Lemma \ref{CS-estimates-vertical} in Section \ref{sec:reg}) and the compact support condition for $\varphi$ and $\eta_1$. 
\end{proof}

Combining the above result with the regularity for $v$ which is deduced in Proposition \ref{prop:reg} in Section \ref{sec:reg} we conclude that the function $v$ satisfies the following equation:

\begin{thm}
\label{prop:equation_2}
{Let $s\in (0,1)$, $n \in \N$}.
{Let $\tilde{a} \in C^2(\overline{\R^{n+1}_+},\R^{(n+1)\times (n+1)})$} be of the form $\tilde{a}(x')= \begin{pmatrix} a(x') & 0 \\ 0 & 1 \end{pmatrix}$ where $a\in {C^2(\R^n,\R^{n\times n}_{sym})}$ is uniformly elliptic. Assume that $\Omega, W\subset \R^n$ are non-empty, bounded and open Lipschitz sets such that $\overline\Omega\cap\overline W=\emptyset$. Let $u_f\in H^s(\R^n)$ be the unique solution of \eqref{eq:fracCald} with exterior value $f\in \tilde H^s(W)$, and let $\tilde u \in H^{1}(\R^{n+1}_+, x_{n+1}^{1-2s})$ be the Caffarelli-Silvestre extension of $u_f$. Consider $v(x'):= \int_0^\infty t^{1-2s}\tilde u(x',t)dt$. 
Then, $v \in H^1(\Omega)$ and
$v$ is a weak solution to 
\begin{align*}
\nabla' \cdot a \nabla' v = 0 \mbox{ in } \Omega,
\end{align*}
i.e., the following identity holds:
\begin{align*}
\int\limits_{\Omega} a \nabla' v \cdot \nabla' \varphi dx' = 0, \ \mbox{ for all } \varphi \in H^1_0(\Omega).
\end{align*}
\end{thm}

\begin{proof}
The result follows immediately from Theorem \ref{prop:equation_1} and by the assumptions on $\tilde u$ and $u_f$, since $v \in H^1(\Omega)$ by Proposition \ref{prop:reg}.
\end{proof}

\section{Density argument for $s\in(0,1)$}\label{sec:density-2}

Given the results on the equation for $v$ in the interior of $\Omega$, we seek to prove that the resulting Cauchy data which are inherited from the nonlocal Calder\'on problem form a dense set in the set of Cauchy data for the local problem. The key observation here will be the following density result:

\begin{prop}
\label{prop:density}
Let $n\geq 3$ and $s\in (0,1)$. Further suppose that $\Omega, W \subset \R^n$ are open, bounded Lipschitz domains with $\overline{\Omega}\cap \overline{W} = \emptyset$.
{Let $\tilde{a} \in C^2(\overline{\R^{n+1}}_{+},\R^{(n+1)\times (n+1)})$} be of the form $\tilde{a}(x')= \begin{pmatrix} a(x') & 0 \\ 0 & 1 \end{pmatrix}$ where $a\in {C^2(\R^n,\R^{n\times n}_{sym})}$ is uniformly elliptic and such that $a\equiv Id$ in $\Omega_e$.
Let $u_f \in  \dot H^{1}(\R^{n+1}_+, x_{n+1}^{1-2s})$ be the unique weak solution of the equation
\begin{align*}
\nabla\cdot x_{n+1}^{1-2s}\tilde a\nabla u_f & = 0 \mbox{ in } \R^{n+1}_+,\\
u_f & = f \mbox{ in } \Omega_e \times \{0\}, \\
\lim\limits_{x_{n+1}\rightarrow 0} x_{n+1}^{1-2s}\p_{n+1} u_f & = 0 \mbox{ in } \Omega,
\end{align*}
with $f\in C_c^{\infty}(W \times \{0\})$. Moreover, define the sets $V\subset H^{1}(\Omega)$ and $V'\subset H^{\frac{1}{2}}(\p\Omega)$ as
\begin{align*}
V:=\left\{ v_f(x'):= \int\limits_{0}^{\infty} t^{1-2s} u_f(x',t) dt, \ f\in {C_c^{\infty}(W)} \right\}, \qquad V':= 
 \{v|_{\p\Omega}, v\in V\}.
\end{align*}
Then,  $\overline {V'} =H^{\frac{1}{2}}(\partial \Omega )$. 
\end{prop}

\subsection{A formal, non-rigorous argument for the density result}
\label{sec:heuristic}
Before turning to the rigorous proof of Proposition \ref{prop:density}, we give an informal outline of its proof only pointing out necessary modifications for a rigorous argument but without discussing the details necessary for the rigorous implementation of the result. A full proof, including the necessary limiting and cut-off arguments is then given in the next subsection. In order to avoid all technical details and to focus on the main idea, in the present heuristic section we present the argument only for the case $a = Id$ and $s= \frac{1}{2}$. The rigorous proof in Section \ref{sec:proof_local} then provides the precise argument for the general case.

\begin{proof}[Sketch of proof of Proposition \ref{prop:density} (for $a=Id, s=\frac{1}{2}$)]
We argue in three steps:\\
\emph{Step 1.} We start by proving the density of $V\subset H^1(\Omega)$ in $S$, where
$$S:=\{v\in H^1(\Omega): \Delta' v = 0 \mbox{ in } \Omega\}.$$
By the Hahn-Banach theorem, it suffices to show that if $\psi \in \tilde{H}^{-1}(\Omega)$ is such that $\psi(v_f)=0$ for all $f\in C_c^{\infty}(W)$, then also $\psi(v)=0$ for all $v\in S$. Here we use the notation $v_f, u_f$ as introduced in the formulation of the proposition.

In order to approach the density result, for $\psi \in \tilde{H}^{-1}(\Omega)$ we introduce an auxiliary problem, which we will refer to as the adjoint problem:
\begin{align}
\label{eq:dual_dense_formal}
\begin{split}
\Delta w & = \psi \mbox{ in } \R^{n+1}_+,\\
w& = 0 \mbox{ in } \Omega_e \times \{0\},\\
\p_{n+1} w & = 0 \mbox{ in } \Omega \times \{0\} .
\end{split}
\end{align}
We will discuss its (weak) solvability in Lemma \ref{well-posed-adjoint} below. Due to the lack of decay of $\psi$ in the normal variable this will rely on an ``explicit'' construction instead of a direct energy argument. 
With such a function $w$ given, we compute 
\begin{equation}\begin{split}\label{eq:density-hb1}
0 & = \psi(v_f) 
=  \langle \psi, \int_{0}^{\infty} u_{f}(\cdot, t) dt \rangle_{\tilde{H}^{-1}(\Omega),H^1(\Omega)}  
\\ & = 
 \int_{\R^n} u_{f} \p_{t}w(x',0)dx'-\int_{\R^{n}}\int_0^\infty \nabla u_{f}\cdot \nabla w dtdx' 
\\ & = \int\limits_{W} f \p_t w(x',0) dx' +\int\limits_{\R^n}  w(x',0)\p_t u_f(x',0) dx' \\ & = \int\limits_{W} f \p_t w(x',0) dx',
\end{split}\end{equation}
since $u_f$ is a weak solution and because of the boundary conditions satisfied by $v_f$ and $u_f$. We remark that on a rigorous level, we will have to insert suitable cut-off functions in order to make use of the equation satisfied by $w$.\\

\emph{Step 2.} By the arbitrary choice of $f\in C^\infty_c(W)$ we now have $\p_t w = 0$ in $W\times \{0\}$. By virtue of the unique continuation property, it then follows that $w\equiv 0$ in $\Omega_e\times \R_+$. In particular, both $w|_{\partial \Omega \times \R_+}$ and $ \p_\nu w|_{\partial \Omega \times \R_+}$ vanish. We use this to conclude the Hahn-Banach argument. Let $v\in S\subset H^1(\Omega)$ and $\beta_1 \in C_c^{\infty}((0,\infty))$ such that $\int\limits_{0}^{\infty} \beta_1(t) dt = 1$, $\beta_1 \geq 0$ and $\supp(\beta_1) \subset (1,2)$. Assume further that $\beta_k(t) = k^{-1} \beta_1(t/k)$. By formula \eqref{eq:normalization}, the support assumptions on $\psi$ and $w$, and using the fact that $v\in S$, we have
\begin{align*}
    -\psi(v) &  =  -\psi \left(\int\limits_{0}^{\infty} \beta_k(t) v dt \right)
    = -\lim\limits_{k\rightarrow \infty} \psi \left(\int\limits_{0}^{\infty} \beta_k(t) v dt \right)\\
&  =  \lim\limits_{k\rightarrow \infty} \left(\int_{\Omega}\int_0^\infty \nabla (v \beta_k) \cdot\nabla w dtdx' - \int_{\p(\Omega\times \R_+)} v \beta_k \p_\nu w dtdx' \right)\\ 
& =
\lim\limits_{k\rightarrow \infty} \left( \int_{\Omega}\nabla'v\cdot\nabla'\left(\int_0^\infty \beta_k w dt \right)dx' + \int\limits_{\Omega}\int\limits_0^\infty v \p_t \beta_k \p_t w dt dx' \right)\\
&= \lim\limits_{k\rightarrow \infty} \left( \int_{\p\Omega}(\nabla'v\cdot\nu' )\left(\int_0^\infty \beta_k w dt \right) dx' + \int\limits_{\Omega}\int\limits_0^\infty v \p_t \beta_k \p_t w dt dx' \right) \\
&  = \lim\limits_{k\rightarrow \infty}  \int\limits_{\Omega} \int\limits_0^\infty v \p_t \beta_k \p_t w dt dx'.
\end{align*}
It thus remains to show that 
\begin{align*}
\lim\limits_{k\rightarrow \infty}  \int\limits_{\Omega} \int\limits_0^\infty v \p_t \beta_k \p_t w dt dx' = 0.
\end{align*}
This follows from the definition of $\beta_k$:
\begin{align*}
\lim\limits_{k\rightarrow \infty}  \int\limits_{\Omega} \int\limits_0^\infty v \p_t \beta_k \p_t w dt dx' = 
\lim\limits_{k\rightarrow \infty} k^{-2} \int\limits_{\Omega} \int\limits_{(k,2k)} v \p_t \beta_1 \p_t w dt dx' = 0.
\end{align*}
This formally proves the density of $V\subset H^1(\Omega)$ in $S$. 
As above, due to the absence of decay, on a rigorous level we need to use bounds for $w$ in order to deduce the claimed limit from above. Moreover, in what follows below, in order to deal with the case $s\in (0,1)$ in a unified way, we will introduce slightly different vertical cut-off functions $\beta_k$. Hence in the rigorous proof below we will use suitable, more careful limiting arguments.\\

\emph{Step 3.} The density result for $V'$ then follows by trace estimates.
\end{proof}

In the next section we make these arguments rigorous and generalize them to $s\in (0,1)$ and to variable coefficient, uniformly elliptic metrics $a$ as defined in the introduction. This includes proving the solvability for \eqref{eq:dual_dense_formal}. In the rigorous implementation of Step 1 this will then necessitate various cut-off and limiting arguments which we present in detail below.

\subsection{A ``local'' proof of the density result of Proposition \ref{prop:density}}
\label{sec:proof_local}

In this section, we present a first rigorous proof of Proposition \ref{prop:density}. To this end, we begin by considering an auxiliary problem, which we will rely on in defining and discussing the adjoint equation in the density proof of Proposition \ref{prop:density}.

\begin{lem}[Solvability of the adjoint problem]\label{well-posed-adjoint}
{Let $s\in (0,1)$, $n \geq 3$}.
{Let $\tilde{a} \in C^2(\overline{\R^{n+1}_+},\R^{(n+1)\times (n+1)})$} be of the form $\tilde{a}(x')= \begin{pmatrix} a(x') & 0 \\ 0 & 1 \end{pmatrix}$ where $a\in {C^2(\R^n,\R^{n\times n}_{sym})}$ is uniformly elliptic.
Let $\psi \in \tilde{H}^{-1}(\Omega)$, and consider the problem 
\begin{align}
\label{eq:dual_dense}
\begin{split}
\nabla\cdot x_{n+1}^{1-2s}\tilde a \nabla w & = x_{n+1}^{1-2s}\psi \mbox{ in } \R^{n+1}_+,\\
w& = 0 \mbox{ in } \Omega_e \times \{0\},\\
\lim\limits_{x_{n+1}\rightarrow 0}x_{n+1}^{1-2s}\p_{n+1} w & = 0 \mbox{ in } \Omega .
\end{split}
\end{align}
The problem \eqref{eq:dual_dense} is solvable in $\dot H^1_{loc,0}(\R^{n+1}_+,x_{n+1}^{1-2s})$, that is, there exists $w\in \dot H^1_{loc}(\R^{n+1}_+,x_{n+1}^{1-2s})$ whose trace vanishes in $\Omega_e\times \{0\}$ and such that
$$
\int\limits_{\R^{n+1}_+} x_{n+1}^{1-2s} \nabla \varphi \cdot\tilde  a\nabla w dx =-\langle \psi, \int\limits_{0}^{\infty} x_{n+1}^{1-2s} \varphi(\cdot, x_{n+1}) dx_{n+1} \rangle_{\tilde{H}^{-1}(\Omega),H^1(\Omega)}$$
for all 
\begin{align*}
\varphi\in H^1_{c,0}(\R^{n+1}_+,x_{n+1}^{1-2s}):=\{v \in H^{1}(\R^{n+1}_+,x_{n+1}^{1-2s}): \ v \mbox{ has compact support in } \overline{\R^{n+1}_+}, \ v|_{\Omega_e\times {\{0\}}} = 0\}.
\end{align*}
\end{lem}

\begin{rmk}
\label{rmk:bdry}
By the regularity of $w$, we may in particular define the (weighted) normal boundary data as a distribution in $\dot{H}^{-s}_{loc}(\R^n)$: Indeed, for $w$ as in Lemma \ref{well-posed-adjoint} and 
\begin{align*}
\varphi  \in H^{1}_{c}(\R^{n+1}_+,x_{n+1}^{1-2s}):=\{\varphi \in H^{1}(\R^{n+1}_+,x_{n+1}^{1-2s}): \ \varphi \mbox{ has compact support in } \overline{\R^{n+1}_+}\},
\end{align*}
we set
\begin{align*}
\int_{\R^n} \varphi(x',0)\lim\limits_{x_{n+1}\rightarrow 0}  x_{n+1}^{1-2s}\p_{n+1}w dx'
& := \int_{\R^{n+1}_+} x_{n+1}^{1-2s} \nabla\varphi\cdot \tilde a\nabla w dx \\ & \quad +  \langle \psi, \int\limits_{0}^{\infty} x_{n+1}^{1-2s} \varphi(\cdot, x_{n+1}) dx_{n+1} \rangle_{\tilde{H}^{-1}(\Omega),H^1(\Omega)},
\end{align*}
and note that by the assumptions on $\varphi$ and the mapping properties of $w$ this indeed yields that $\lim\limits_{x_{n+1}\rightarrow 0}  x_{n+1}^{1-2s}\p_{n+1}w \in \dot{H}^{-s}_{loc}(\R^n)$.
Using this definition, in the proof of Proposition \ref{prop:density}, for suitable choices of $\varphi,v$ (which will be built from $w$), we will often consider the following bilinear form:
$$B(v,\varphi):= -\int_{\R^{n+1}_+} x_{n+1}^{1-2s} \nabla\varphi\cdot\tilde  a\nabla v dx + \int_{\R^n}   \varphi(x',0) \lim\limits_{x_{n+1}\rightarrow 0} x_{n+1}^{1-2s}\p_{n+1}vdx'.$$
\end{rmk}

\begin{proof}[Proof of Lemma \ref{well-posed-adjoint}]
We start by observing that for $n\geq 3$ the $n$-dimensional problem
\begin{align*}
\nabla'\cdot a\nabla' u_1 & = \psi \mbox{ in } \R^{n}
\end{align*}
is solvable in the space $L^{\frac{2n}{n-2}}(\R^n)\cap \dot{H}^{1}(\R^n)$ via standard energy estimates. We define the function $\tilde{u}_1 \in L^{\infty}(\R,L^{\frac{2n}{n-2}}(\R^n))\cap L^{\infty}(\R,\dot{H}^{1}(\R^n))$ to be constant in the vertical direction, i.e., $\tilde{u}_1(x',x_{n+1}):= u_1(x')$. Building on this, we next consider the problem
\begin{align}
\label{eq:dual_aux_2}
\begin{split}
\nabla\cdot x_{n+1}^{1-2s}\tilde a \nabla u_2 & = 0 \mbox{ in } \R^{n+1}_+,\\
u_2 & = 0 \mbox{ on } \Omega_e \times \{0\},\\
\lim\limits_{x_{n+1}\rightarrow 0}x_{n+1}^{1-2s}\p_{n+1} u_2 & = \lim\limits_{x_{n+1}\rightarrow 0}x_{n+1}^{1-2s}\p_{n+1} P_su_1 \mbox{ on } \Omega ,
\end{split}
\end{align}
where $P_s$ denotes the Caffarelli-Silvestre extension operator (c.f. \cite{ST10} and \eqref{eq:CS_Stinga}) extending functions from $\R^n$ to $\R^{n+1}_+$ corresponding to the operator $\nabla\cdot x_{n+1}^{1-2s}\tilde a \nabla$. 

In order to deduce the existence of an $\dot{H}^1(\R^{n+1}_+, x_{n+1}^{1-2s})$ solution to \eqref{eq:dual_aux_2}, we next show that $\lim\limits_{x_{n+1}\rightarrow 0}x_{n+1}^{1-2s}\p_{n+1} P_su_1  \in H^{-s}(\Omega)$. Indeed, this follows from the fact that $u_1 \in \dot{H}^1(\R^n)\cap L^{\frac{2n}{n-2}}(\R^n)$: We consider the splitting
\begin{align}\label{eq:splitting}
u_1 = \eta_{B_R} u_1 + (1-\eta_{B_R})m_{h}(D) u_1 + (1-\eta_{B_R}) m_{\ell}(D) u_1,
\end{align}
where $R>0$ is such that $\overline{\Omega} \subset B_R$, and $\eta_{B_R}$ is a smooth cut-off function supported in $B_{2R}$ which equals one in $B_R$.
The functions $m_{\ell}(D)$ and $m_h(D)$ are a low and a high frequency projection, i.e., for $f\in \mathcal{S}'(\R^n)$ we set $$m_{\ell}(D) f = \F^{-1} (\chi_{B_2} \F f), \qquad m_{h}(D) f = \F^{-1} ((1-
\chi_{B_2}) \F f),$$ where $\chi_{B_2}$ is a smooth cut-off function supported in $B_4$ which equals one in $B_2$.  Let $g$ indicate either $\eta_{B_R}u_1$ or $(1-\eta_{B_R})m_{h}(D) u_1$. Since $u_1 \in  
\dot{H}^1(\R^n)\cap L^{\frac{2n}{n-2}}(\R^n)$, we immediately obtain that $g \in H^1(\R^n)$, and thus  $g\in H^s(\R^n)$. As a consequence, by the trace estimate 
$\dot{H}^1(\R^{n+1}_+, x_{n+1}^{1-2s}) \hookrightarrow \dot{H}^s(\R^n)$ and by energy estimates, $P_sg \in \dot{H}^1(\R^{n+1}_+, x_{n+1}^{1-2s})$. Thus, $\lim\limits_{x_{n+1}\rightarrow 0}x_{n+1}^{1-2s}\p_{n+1} P_s g \in \dot{H}^{-s}(\R^n)$, 
and in particular $\lim\limits_{x_{n+1}\rightarrow 0}x_{n+1}^{1-2s}\p_{n+1} P_s g \in H^{-s}(\Omega)$. For the remaining contribution on the right hand side of \eqref{eq:splitting}, we use Lemma \ref{CS-Lp} and 
the fact that $m_\ell(D)$ is a Fourier multiplier mapping $L^p$ into itself for all $p\in(1,\infty)$ in order to deduce that $\lim\limits_{x_{n+1}\rightarrow 0}x_{n+1}^{1-2s}\p_{n+1} P_s((1-\eta_{B_R}) m_{\ell}(D) u_1) \in L^2(\Omega)$. This eventually implies the claim $\lim\limits_{x_{n+1}\rightarrow 0}x_{n+1}^{1-2s}\p_{n+1} P_s u_1 \in H^{-s}(\Omega)$. 

We use this fact to discuss the solvability of \eqref{eq:dual_aux_2}. By the trace estimate $\dot{H}^1(\R^{n+1}_+, x_{n+1}^{1-2s})  \hookrightarrow \dot{H}^s(\R^n)$ the functional 
\begin{align*}
\dot{H}^1(\R^{n+1}_+, x_{n+1}^{1-2s}) \ni \varphi \mapsto \int\limits_{\Omega}\varphi \lim\limits_{x_{n+1}\rightarrow 0}x_{n+1}^{1-2s}\p_{n+1} P_su_1  dx' 
\end{align*}
is bounded, and problem \eqref{eq:dual_aux_2} has a unique solution $u_2 \in \dot{H}^{1}(\R^{n+1}_+, x_{n+1}^{1-2s})$ satisfying the required vanishing exterior data by standard energy estimates. Moreover, by the Sobolev-trace inequality, we obtain that $u_2(\cdot,0) \in L^{\frac{2n}{n-2s}}(\R^n)$. The function $u_2$ can hence also be viewed as $P_s(u_2(\cdot,0))$, that is as the extension of the $ L^{\frac{2n}{n-2s}}(\R^n)$ function $u_2(\cdot,0)$ with respect to the operator $\nabla\cdot x_{n+1}^{1-2s}\tilde a \nabla$ in $\R^{n+1}_+$. \\

We will now show that $w := \tilde u_1-P_s u_1+u_2$ is a $\dot H^1_{loc,0}(\R^{n+1}_+,x_{n+1}^{1-2s})$ solution of problem \eqref{eq:dual_dense}. To this end, we observe that for any $\varphi\in  H^1_{c,0}(\R^{n+1}_+,x_{n+1}^{1-2s})$ it holds
\begin{align*}
    \int_{\R^{n+1}_+} \nabla\varphi\cdot x_{n+1}^{1-2s}\tilde a\nabla P_su_1 dx &= \int_{\R^n} \varphi(x',0) \lim\limits_{x_{n+1}\rightarrow 0} x_{n+1}^{1-2s}\p_{n+1}P_su_1dx' , \\
    \int_{\R^{n+1}_+} \nabla\varphi\cdot x_{n+1}^{1-2s}\tilde a\nabla u_2 dx
     &= \int_{\Omega} \varphi(x',0)\lim\limits_{x_{n+1}\rightarrow 0} x_{n+1}^{1-2s}\p_{n+1}P_su_1dx' \\
    & \quad + \int_{\Omega_e} \varphi(x',0)\lim\limits_{x_{n+1}\rightarrow 0} x_{n+1}^{1-2s}\p_{n+1}u_2dx' .
\end{align*}
Moreover, we have
\begin{align*}
\int_{\R^{n+1}_+} \nabla\varphi\cdot x_{n+1}^{1-2s}\tilde a\nabla \tilde u_1 dx & = \int_{\R^{n+1}_+} x_{n+1}^{1-2s}\nabla'\varphi\cdot a(x')\nabla'  u_1 dx \\ & = \int_{\R^{n}}\nabla'\left(\int_0^\infty x_{n+1}^{1-2s}\varphi dx_{n+1}\right)\cdot a(x')\nabla'  u_1dx' 
\\ & = -\langle \psi, \int_{0}^{\infty} x_{n+1}^{1-2s} \varphi(\cdot, x_{n+1}) dx_{n+1} \rangle_{\tilde{H}^{-1}(\Omega),H^1(\Omega)}.
\end{align*}
Here we have used that $ \int_0^\infty x_{n+1}^{1-2s}\varphi(\cdot,x_{n+1}) dx_{n+1} \in H^1(\R^n)$
since, by the compactness of the support of $\varphi$, for $M>0$ large enough it holds
\begin{align*}
    \left\| \int_0^\infty x_{n+1}^{1-2s}\varphi(\cdot,x_{n+1}) dx_{n+1} \right\|^2_{\dot H^1(\R^n)} & = \int_{\R^n} \left|\nabla' \left( \int_0^\infty x_{n+1}^{1-2s}\varphi dx_{n+1} \right) \right|^2dx' 
    \\ & \leq \int_{\R^n} \left|\int_0^M x_{n+1}^{1-2s}\nabla'\varphi dx_{n+1} \right|^2dx'
    \\ & \leq \int_{\R^n} \left(\int_0^M x_{n+1}^{1-2s}dx_{n+1}\right)\left(\int_0^M x_{n+1}^{1-2s}|\nabla'\varphi|^2 dx_{n+1}\right) dx'
    \\ & \leq \frac{M^{2-2s}}{2-2s} \int_{\R^{n+1}_+} x_{n+1}^{1-2s}|\nabla'\varphi|^2 dx
    \\ & \lesssim \|\varphi\|^2_{ H^1(\R^{n+1}_+,x_{n+1}^{1-2s})}<\infty,
\end{align*}
and similarly for the $L^2(\R^n)$ norm. The desired result now follows by combining the above computations for the components $\tilde u_1, P_su_1$ and $u_2$ of the candidate solution $w$.
\end{proof}

Having fixed the ideas of our first proof of Proposition \ref{prop:density} in Section \ref{sec:heuristic}, we now turn to making them rigorous. To this end, in what follows we will carry out the relevant approximation and cut-off steps in detail.

\begin{proof}[Proof of Proposition \ref{prop:density}]
We start by proving the density of $V\subset H^1(\Omega)$ in $S$, where
$$S:=\{v\in H^1(\Omega): \nabla'\cdot a\nabla' v = 0 \mbox{ in } \Omega\}.$$
By the Hahn-Banach theorem, it suffices to show that if $\psi \in \tilde{H}^{-1}(\Omega)$ is such that $\psi(v_f)=0$ for all $f\in C_c^{\infty}(W)$, then also $\psi(v)=0$ for all $v\in S$. Here we use the notations $v_f, u_f$ as introduced in the formulation of the proposition.\\

\emph{Step 1a: Setting up the duality argument.} 
In order to approach the density result, we consider the auxiliary problem \eqref{eq:dual_dense}, which is solvable by virtue of Lemma \ref{well-posed-adjoint}. {In order to make use of it, we have to consider test functions with compact support. To this end, in what follows we first introduce two cut-off functions, one for the vertical, one for the tangential directions.}
Let $\eta_k(t) := \eta_1(t/k)$, where $\eta_1\in C^\infty_c([0,2])$ is a smooth cut-off function satisfying $\eta_1\equiv 1$ in a neighbourhood of {$t=0$} and $\int\limits_{0}^{\infty} t^{1-2s}\eta_1  dt= 1$. Observe that in particular
\begin{equation}\label{eq:normalization}
    k^{2s-2}\int\limits_0^\infty t^{1-2s}\eta_k(t)dt = k^{2s-2}\int\limits_0^\infty t^{1-2s}\eta_1(t/k)dt = \int\limits_0^\infty t^{1-2s}\eta_1(t)dt =1.
\end{equation}
Moreover, if $R>0$ is so large that $\overline{\Omega}\cup \overline{W}\subset B_R$, let $\sigma_k(x'):= \sigma_1(x'/k)$, where $\sigma_1\in C^\infty_c(B_{2R})$ is a smooth and radial cut-off function satisfying $\sigma_1\equiv 1$ in $B_R$. Observe that for all $f\in C^\infty_c(W)$, the function $u_{f,k}(x):=u_f(x)\sigma_k(x')\eta_k(x_{n+1})$ belongs to $\dot H^1_c(\R^{n+1}_+,x_{n+1}^{1-2s})$, and is thus an admissible test function for the adjoint problem \eqref{eq:dual_dense} (see Remark \ref{rmk:bdry}). Therefore, using the notation from the proof of Lemma \ref{well-posed-adjoint} and denoting by $w:= \tilde{u}_1 -P_su_1 + u_2$ the solution constructed in Lemma \ref{well-posed-adjoint}, we infer
\begin{equation}\begin{split}\label{eq:density-hb}
0 & = \psi(v_f) 
= \lim\limits_{k\rightarrow \infty} \langle \psi, \int_{0}^{\infty} t^{1-2s} u_{f,k}(\cdot, t) dt \rangle_{\tilde{H}^{-1}(\Omega),H^1(\Omega)} = \lim\limits_{k\rightarrow \infty} B(w,u_{f,k})  
\\ & = 
\lim\limits_{k\rightarrow \infty} \left( \int_{\R^n} u_{f,k} \lim\limits_{t\rightarrow 0} t^{1-2s}\p_{t}wdx'-\int_{\R^{n}}\int_0^\infty t^{1-2s}\nabla u_{f,k}\cdot \tilde a\nabla w dtdx' \right)
\\ & = \int\limits_{W} f \lim\limits_{t\rightarrow 0}(t^{1-2s}\p_t w) dx' +\lim\limits_{k\rightarrow \infty} I_k,
\end{split}\end{equation}
where the bulk contributions $I_k$ are given by
\begin{align}
\label{eq:Ik}
    I_k :&= -\int\limits_{\R^n} \int\limits_{0}^{\infty} t^{1-2s}(\sigma_k\eta_k\nabla w) \cdot \tilde a\nabla u_f dtdx'  -\int\limits_{\R^n} \int\limits_{0}^{\infty}  t^{1-2s}u_f \nabla (\sigma_k\eta_k) \cdot {\tilde{a}} \nabla w dtdx' .
\end{align}
The second equality in formula \eqref{eq:density-hb} holds by the $\tilde{H}^{-1}(\Omega)-H^{1}(\Omega)$ duality and by the fact that
\begin{align*}
\int\limits_{0}^{\infty} t^{1-2s}\sigma_k(x')\eta_k(t) u_f(x',t) dt \rightarrow \int\limits_{0}^{\infty} t^{1-2s} u_f(x',t) dt \mbox{ in } H^1(\Omega) \quad\mbox{ as } k \rightarrow \infty.
\end{align*}
Indeed, it holds that
\begin{align*}
&\left\| \int\limits_{0}^{\infty} t^{1-2s}\sigma_k(x')\eta_k(t) u_f(x',t) dt - \int\limits_{0}^{\infty} t^{1-2s} u_f(x',t) dt \right\|_{H^1(\Omega)}\\
&\leq \left\| \int\limits_{k}^{2k} t^{1-2s} (|u_f(\cdot, t)| + |\nabla u_f(\cdot, t)|) dt\right\|_{L^2(\Omega)}\\
&\leq \int\limits_{k}^{2k} t^{1-2s-n} \|u_f(\cdot, 0)\|_{L^1(\R^n)} dt  \lesssim  k^{1-2s-n} \|f\|_{L^2(\Omega)} \rightarrow 0 \mbox{ as } k \rightarrow \infty.
\end{align*}
In the above estimate, we used formula \eqref{eq:CS-decay} from Lemma \ref{CS-estimates-vertical}, the fact that $u_f(\cdot,0)$ has compact support, and the assumption $n\geq 3$.\\

\emph{Step 1b: Decay estimates for the error contributions.} Next, we seek to show that $I_k\rightarrow 0$ as $k\rightarrow\infty$. Integrating by parts the second term on the right hand side in the definition \eqref{eq:Ik} of $I_k$, and using the fact that $u_f$ is a weak solution, we get
\begin{align*}
    I_k &= -\int\limits_{\R^n} \int\limits_{0}^{\infty} t^{1-2s}\nabla(\sigma_k\eta_k w) \cdot \tilde a\nabla u_f dtdx' + \int\limits_{\R^n} \int\limits_{0}^{\infty} t^{1-2s}w\nabla(\sigma_k\eta_k) \cdot \tilde a\nabla u_f dtdx' \\ 
    & \quad +\int\limits_{\R^n} \int\limits_{0}^{\infty}  w\nabla\cdot({\tilde{a}} t^{1-2s}u_f \nabla (\sigma_k\eta_k)) dtdx' 
    \\ & = \int\limits_{\R^n} \sigma_k w(x',0)\lim\limits_{t\rightarrow 0} t^{1-2s}\p_t u_f dx' + 2\int\limits_{\R^n} \int\limits_{0}^{\infty} wt^{1-2s}\nabla(\sigma_k\eta_k) \cdot {\tilde{a}}\nabla u_f dtdx' \\ 
    & \quad +\int\limits_{\R^n} \int\limits_{0}^{\infty}  wt^{1-2s}u_f {\tilde L}(\sigma_k\eta_k) dtdx' +(1-2s)\int\limits_{\R^n} \int\limits_{0}^{\infty}  wt^{-2s}u_f \sigma_k\p_t\eta_k dtdx'
    \\ & = \int\limits_{B_{2Rk}} \int\limits_{0}^{2k} w t^{1-2s} \left( 2 \nabla(\sigma_k\eta_k)\cdot \tilde{a} \nabla u_f  + \left({\tilde L} (\sigma_k\eta_k) + \frac{1-2s}{t} \sigma_k  \p_t \eta_k\right)u_f  \right) dtdx'.
\end{align*}
{Here we have used the notation $\tilde L:= \nabla \cdot \tilde{a} \nabla$.}
The boundary term in the second step of the above computation is well-defined {and vanishes}, since by Lemma \ref{well-posed-adjoint} {the supports of $w(x',0)$, $\lim\limits_{t\rightarrow 0}(t^{1-2s}\p_t u_f)$ are disjoint and, moreover,} 
\begin{align*}
   \left| \int\limits_{\R^n}  \sigma_k w(x',0)\lim\limits_{t\rightarrow 0}(t^{1-2s}\p_t u_f) dx' \right| & \lesssim \|(\nabla'\cdot a\nabla')^s u_f(\cdot,0)\|_{H^{-s}(\Omega)}\|w(\cdot,0)\|_{H^{s}(\Omega)} <\infty.
\end{align*}  
Here we have used that $w(x',0)= \tilde u_1(x') - P_s u_1(x',0) + u_2(x',0)$, with $\tilde u_1\in \dot{H}^1(\R^n)\cap L^{\frac{2n}{n-2}}(\R^n)$ and $P_s u_1(\cdot, 0), u_2(\cdot,0) \in H^s_{loc}(\R^n)\cap L^{\frac{2n}{n-2s}}(\R^n)$.

We are {now} ready to estimate the bulk terms $I_k$. Let $A_k := B_{2Rk}\setminus B_{Rk}$, {where $R>0$ denotes the radius from Step 1a.} Using the boundedness of $\tilde{a}, \nabla \tilde{a}$, we compute
\begin{align*}
    |I_k| & \lesssim  \int\limits_{B_{2Rk}} \int\limits_{0}^{2k} t^{1-2s} |w| \left(  |\nabla(\sigma_k\eta_k)|\, |\nabla u_f|  + \left(|{\tilde L} (\sigma_k\eta_k)| + t^{-1} |\sigma_k| \,  |\p_t \eta_k|\right)|u_f|  \right) dtdx'
    \\ & \lesssim
    \int\limits_{B_{2Rk}} \int\limits_{0}^{2k} t^{1-2s} |w| \left(  (|\nabla'\sigma_k| + |\p_t\eta_k|)\, |\nabla u_f|  + \left(|{\tilde L}\sigma_k| + |\p_t^2\eta_k| + t^{-1} |\p_t \eta_k|\right)|u_f|  \right) dtdx'
     \\ & \lesssim
    k^{-1}\int\limits_{\R^n} \int\limits_{k}^{2k} t^{1-2s} |w| \left(  |\nabla u_f|  + \frac{|u_f|}{k}  \right) dtdx' + k^{-1}\int\limits_{A_k} \int\limits_{0}^{2k} t^{1-2s} |w| \left( |\nabla u_f|  + \frac{|u_f|}{k}  \right) dtdx' \\ & =: I_{1,k}(w)+ I_{2,k}(w).
\end{align*}

We estimate the contributions in $I_{1,k}(w)$ and $I_{2,k}(w)$ separately.

\medskip

\emph{Step 1 b(i): Estimate for $I_{1,k}(w)$.} We begin by considering the term $I_{1,k}(w)$. Using the triangle inequality, we first deduce bounds for the expression
\begin{align*}
\int\limits_{\R^n} |\tilde w(x',t)| |\nabla^j u_f(x',t)|  dx',
\end{align*}
where $\tilde{w}$ denotes any of the functions $u_1, P_su_1, u_2$. As the functions $P_s u_1, u_2$ have decay in the vertical direction, while the function $u_1$ does not have such decay, we split the proof again into two parts, discussing first the case $\tilde{w} = u_1$ and then the cases $\tilde{w} =  P_su_1, u_2$.

\smallskip

\emph{The case $\tilde{w}=u_1$}.
Let us first denote by $\tilde w$ the function $u_1$, and assume $j\in\{0,1\}$. Then by formula \eqref{eq:extension_est}
\begin{align*}
\int\limits_{\R^n} |\tilde w(x',t)| |\nabla^j u_f(x',t)|  dx' & \leq \|\tilde w \|_{L^{r_1}(\R^n)}\|\nabla^j u_f(\cdot, t)\|_{L^{r_2}(\R^n)} 
\\ & \lesssim  t^{n(\frac{1}{p_2}-1)-j} \|\tilde w\|_{L^{r_1}(\R^n)} \|u_f(\cdot,0)\|_{L^{q_2}(\R^n)},
\end{align*}
where 
\begin{align}
\label{eq:conjug_cond}
\frac{1}{r_1} + \frac{1}{r_2} = 1, \qquad \mbox{ and } \qquad  1+\frac{1}{r_2} = \frac{1}{p_2} + \frac{1}{q_2}.
\end{align}
Using that $u_1 \in L^{\frac{2n}{n-2}}(\R^n)$, we choose $r_1 = \frac{2n}{n-2}$ and $r_2 = \frac{2n}{n+2}$.
In order to obtain the maximal possible decay, we then choose $q_2 = 1$ (for which we use that $u_f(\cdot,0)\in \tilde{H}^s(\Omega\cup W) \subset L^1(\Omega\cup W)$ by Hölder's inequality). As a consequence,
\begin{equation*}\begin{split}
|I_{1,k}(\tilde w)| & \lesssim k^{-1} \|\tilde w\|_{L^{r_1}(\R^n)} \|f\|_{H^{s}(W)} \int\limits_{k}^{2k} t^{-2s+n(\frac{1}{p_2}-1)} dt \\
& \leq k^{-2s+n(\frac{1}{p_2}-1)}  \|\tilde w\|_{L^{r_1}(\R^n)} \|f\|_{H^{s}(W)}.
\end{split}\end{equation*}
Since $\frac{1}{p_2}-1\leq 0$, we thus infer that $I_{1,k}(\tilde w)$ vanishes as $k\rightarrow\infty$. 

\smallskip

\emph{The case $\tilde{w}=P_s u_1, u_2$}.
Let us next denote by $\tilde w$ any of the functions $ P_su_1, u_2$, and assume $j\in\{0,1\}$. Then by formula \eqref{eq:extension_est}
\begin{align*}
\int\limits_{\R^n} |\tilde w(x',t)| |\nabla^j u_f(x',t)|  dx' & \leq \|\tilde w(\cdot, t)\|_{L^{r_1}(\R^n)}\|\nabla^j u_f(\cdot, t)\|_{L^{r_2}(\R^n)} 
\\ & \lesssim t^{n(\frac{1}{p_1}-1)} t^{n(\frac{1}{p_2}-1)-j} \|\tilde w(\cdot,0)\|_{L^{q_1}(\R^n)} \|u_f(\cdot,0)\|_{L^{q_2}(\R^n)}
\\ & = t^{n(1-\frac{1}{q_1}-\frac{1}{q_2})-j} \|\tilde w(\cdot,0)\|_{L^{q_1}(\R^n)} \|u_f(\cdot,0)\|_{L^{q_2}(\R^n)},
\end{align*}
where 
\begin{align}
\frac{1}{r_1} + \frac{1}{r_2} = 1, \qquad \mbox{ and } \qquad  1+\frac{1}{r_i} = \frac{1}{p_i} + \frac{1}{q_i} \quad \mbox{ for } i\in\{ 1,2\}.
\end{align}
In order to obtain the maximal possible decay, we again choose $q_2 = 1$. This yields decay of the form $t^{-n/q_1-j}$, and thus
\begin{equation*}
\begin{split}
|I_{1,k}(\tilde w)| & \lesssim k^{-1} \|\tilde w(\cdot,0)\|_{L^{q_1}(\R^n)} \|f\|_{H^{s}(W)} \int\limits_{k}^{2k} t^{-2s-n/q_1} dt \leq k^{-2s-n/q_1} \|\tilde w(\cdot,0)\|_{L^{q_1}(\R^n)} \|f\|_{H^{s}(W)}.
\end{split}\end{equation*}
By choosing $q_1= \frac{2n}{n-2s}$ or $q_1 = \frac{2n}{n-2}$, respectively, we ensure that $\|\tilde w(\cdot,0)\|_{L^{q_1}(\R^n)}<\infty$ (for $u_2, P_s u_1$, respectively), and thus $I_{1,k}(\tilde w)$ vanishes as $k\rightarrow\infty$. As a result of the last two estimates, we have obtained that $I_{1,k}(w)$ itself vanishes as $k\rightarrow\infty$.

\smallskip

\emph{Step 1b(ii): Estimate for $I_{2,k}(w)$.}
In order to estimate the last term $I_{2,k}(w)$, we compute as in Lemma \ref{CS-estimates-vertical} ({borrowing the kernel notation $K_{0,t}(x'):= \frac{t^{2s}}{(|x'|^2 + t^2)^{\frac{n}{2} +s}}$ from there})
\begin{align*}
\|u_f(\cdot, t)\|^r_{L^{r}(A_k)} & \lesssim \int\limits_{A_k} |(|u_f(\cdot,0)|\ast K_{0,t})(x')|^r dx'  \\ &  \lesssim t^{2sr}\int\limits_{A_k} \left(\int\limits_{\Omega\cup W} \frac{|u_f(z,0)|}{(|x'-z|^2 + t^{2})^{\frac{n}{2} +s}}dz \right)^rdx' \\ & \lesssim \frac{t^{2sr}k^n}{(k^2 + t^{2})^{r(\frac{n}{2} +s)}}\|u_f(\cdot,0)\|^r_{L^1(\R^n)},
\end{align*}
and similarly
\begin{align*}
\|\nabla u_f(\cdot, t)\|^r_{L^{r}(A_k)} \lesssim \frac{t^{(2s-1)r} k^n}{(k^2 + t^{2})^{r(\frac{n}{2} +s)}} \|u_f(\cdot,0)\|_{L^1(\R^n)}^r.
\end{align*}
We now again split the discussion of the estimate into two cases by the triangle inequality.

\smallskip

\emph{The case $\tilde{w}=u_1$}.
In the case that $\tilde{w} = u_1$, by the above computation, we have for $j\in\{0,1\}$ and $r_1 = \frac{2n}{n-2}$
\begin{align*}
\int\limits_{A_k} |\tilde w(x',t)| |\nabla^j u_f(x',t)|  dx' & \leq \|\tilde w\|_{L^{r_1}(\R^n)}\| \nabla^j u_f(\cdot, t)\|_{L^{r_2}(A_k)} 
\\ & \lesssim \frac{t^{2s-j}k^{n/r_2}}{(k^2 + t^{2})^{\frac{n}{2} +s}} \|\tilde w\|_{L^{r_1}(\R^n)} \|u_f(\cdot,0)\|_{L^1(\R^n)},
\end{align*}
which leads to
\begin{align*}
 &   k^{j-2}\int\limits_{A_k} \int\limits_{0}^{2k} t^{1-2s} |\tilde w(x',t)| |\nabla^j u_f(x',t)|  dtdx' \\
 &  \lesssim  \|u_1 \|_{L^{r_1}(\R^n)} \|u_f(\cdot,0)\|_{L^1(\R^n)} k^{j-2+\frac{n}{r_2}}  \int\limits_{0}^{2k}  \frac{t^{1-j}}{(k^2 + t^{2})^{\frac{n}{2} +s}}  dt 
    \\ & \lesssim  \|u_1 \|_{L^{r_1}(\R^n)} \|f\|_{H^s(W)} k^{-n-2s+ \frac{n}{r_2}} \int\limits_{0}^{2}  \frac{\tau^{1-j}}{(1 + \tau^{2})^{\frac{n}{2} +s}}  d\tau .
\end{align*}
Using that $r_1,r_2 $ are dual exponents and, hence, $r_2 = \frac{2n}{n+2}$, we observe that the above term is finite for all $k>0$, and has a decay of the form $k^{-2s + 1 -\frac{n}{2}}$. Since $-2s + 1 -\frac{n}{2}<0$ for $n\geq 3$, the contribution $I_{2,k}(\tilde w)$ vanishes as $k\rightarrow\infty$.\\

\smallskip

\emph{The case $\tilde{w} = P_s u_1, u_2$.}
As above, by the decay estimates from \eqref{eq:extension_est}, we have for $j\in\{0,1\}$
\begin{align*}
\int\limits_{A_k} |\tilde w(x',t)| |\nabla^j u_f(x',t)|  dx' & \leq \|\tilde w(\cdot, t)\|_{L^{r_1}(\R^n)}\| \nabla^j u_f(\cdot, t)\|_{L^{r_2}(A_k)} 
\\ & \lesssim t^{n(\frac{1}{p_1}-1)}\frac{t^{2s-j}k^{n/r_2}}{(k^2 + t^{2})^{\frac{n}{2} +s}} \|\tilde w(\cdot,0)\|_{L^{q_1}(\R^n)} \|u_f(\cdot,0)\|_{L^1(\R^n)},
\end{align*}
which leads to
\begin{align*}
   & k^{j-2}\int\limits_{A_k} \int\limits_{0}^{2k} t^{1-2s} |\tilde w(x',t)| |\nabla^j u_f(x',t)|  dtdx' \\
    & \lesssim  \|\tilde w(\cdot,0)\|_{L^{q_1}(\R^n)} \|u_f(\cdot,0)\|_{L^1(\R^n)} k^{j-2+\frac{n}{r_2}}  \int\limits_{0}^{2k}  \frac{t^{n(\frac{1}{p_1}-1)+1-j}}{(k^2 + t^{2})^{\frac{n}{2} +s}}  dt 
    \\ & \lesssim  \|\tilde w(\cdot,0)\|_{L^{q_1}(\R^n)} \|f\|_{H^s(W)} k^{-\frac{n}{q_1}-2s} \int\limits_{0}^{2}  \frac{\tau^{n(\frac{1}{p_1}-1)+1-j}}{(1 + \tau^{2})^{\frac{n}{2} +s}}  d\tau .
\end{align*}
Recalling that the exponents $r_1, r_2, p_1, p_2, q_1,q_2$ satisfy the same constraints as in \eqref{eq:conjug_cond} above, and by choosing $p_1=1$, and $q_1$ as for $I_{1,k}$, we observe that the above term is finite for all $k>0$, and has a decay of the form $k^{-\frac{n}{q_1}-2s}$. Thus $I_{2,k}(w)$ vanishes as $k\rightarrow\infty$.\\

\emph{Step 2: Unique continuation and conclusion of the Hahn-Banach argument.} We are now left with $\int\limits_{W} f \lim\limits_{t\rightarrow 0}(t^{1-2s}\p_t w) dx'=0$, which by the arbitrary choice of $f\in C^\infty_c(W)$ gives $\lim\limits_{t\rightarrow 0}t^{1-2s}\p_t w = 0$ in $W\times \{0\}$. Since also $w = 0$ in $W \times \{0\}$, by virtue of the unique continuation property, it then follows that $w\equiv 0$ in $\Omega_e\times \R_+$. In particular, both $w|_{\partial \Omega \times \R_+}$ and $ \p_\nu w|_{\partial \Omega \times \R_+}$ vanish as distributions, and $\lim\limits_{x_{n+1}\rightarrow 0}x_{n+1}^{1-2s}\p_{n+1}w \equiv 0$ in $\R^n$. \\

We use this to conclude the Hahn-Banach argument. Let $v\in S\subset H^1(\Omega)$, and fix an extension $Ev \in H^1(\R^n)$ whose support is contained in a bounded open set $\Omega'\supset \Omega$. 
In order to avoid dealing with boundary terms on $\R^n \times \{0\}$, we introduce a further cut-off function.  Given $b\in { (0,1)}$, consider a smooth function $\gamma_{b}: \R_+\rightarrow {(0,b)}$ such that supp$(\gamma_{b})\subseteq [0,\frac{2-b}{1-b}]$ and $\gamma_{b}(t)\equiv b$ for $t\in [1,\frac{1}{1-b}]$. It is easily shown that one can assume $\int_0^\infty \gamma_{b}(t)dt=\frac{b}{1-b}$ and $|\nabla^\ell \gamma_{b}|\leq C$, where $\ell\in\{0,1,2\}$ and $C>1$ is independent of $b$. 
Observe that for all $k\in\N$
$$I_{b,k}:=\int_0^\infty t^{1-2s}\gamma_{b}(t-k)dt = \int_0^{\frac{2-b}{1-b}} (t+k)^{1-2s}\gamma_{b}(t)dt, $$
where $I_{b,k}$ depends continuously on the parameter $b$, and therefore
\begin{equation*}
I_{b,k} \in \begin{cases}
\frac{b\, k^{1-2s}}{1-b}\left(1 , \left(1+\frac{2-b}{k(1-b)}\right)^{1-2s} \right), &\text{if $s\in (0,\frac{1}{2}]$},\\
\frac{b\, k^{1-2s}}{1-b}\left( \left(1+\frac{2-b}{k(1-b)}\right)^{1-2s}, 1 \right), &\text{if $s\in (\frac{1}{2},1)$}.
\end{cases}
\end{equation*}
In both cases we see that for $b\in {(0,1)}$ the values reached by $I_{b,k}$ can get both arbitrarily large and arbitrarily close to $0$. Thus by continuity for all $k\in\N$ we can find $b_{k,s}\in (0,1)$ such that $I_{b_{k,s},k}=1$. Let now $\beta_k(t):= \gamma_{b_{k,s}}(t-k)$ and $R_{k,s}:= k+\frac{1}{1-b_{k,s}}$. By the above construction, $\beta_k: (0,\infty) \rightarrow [0,1]$ verifies
\begin{align*}
&\mbox{supp}(\beta_k) \subseteq (k, R_{k,s}+1),\\
&\beta_k(t) = b_{k,s} \mbox{ for } t\in (k+1,R_{k,s}),\\
& |\nabla^{\ell} \beta_k(t)| \leq C,\\
& \int\limits_{0}^{\infty} t^{1-2s}\beta_k(t) dt = 1,
\end{align*}
where $\ell \in \{0,1,2\}$ and {the constant} $C>1$ is independent of $k$.

By formula \eqref{eq:normalization} and the support assumption on $\psi$, for all $k\in\N$ we have
$$ \psi(v) = \langle \psi, v \int_0^\infty t^{1-2s}\beta_k dt \rangle_{\tilde H^{-1}(\Omega),H^1(\Omega)} =\langle \psi, \int_0^\infty t^{1-2s}\beta_k(t)\,Ev(\cdot) dt \rangle_{\tilde H^{-1}(\Omega),H^1(\Omega)}.$$
Since $\beta_k \,Ev\in \dot H^1_c(\R^{n+1}_+,x_{n+1}^{1-2s})$, we further deduce $\psi(v) =  B(w,\beta_k \,Ev)$. Therefore, using the support information on $w$ and observing that the boundary conditions on $\R^n \times \{0\}$ vanish due to the support conditions for $\beta_k$, we obtain
\begin{align*}
    \psi(v) & =   - \int_{\R^{n}}\int_0^\infty t^{1-2s}\nabla(\beta_k \, Ev)\cdot\tilde  a\nabla w dtdx' 
    \\ & = 
    - \left( \int_{\R^{n}}\int_0^\infty t^{1-2s}Ev\,\p_{t}\beta_k \p_{t} w dtdx' + \int_{\R^{n}}a\nabla'(Ev)\cdot\nabla'\left(\int_0^\infty t^{1-2s}\beta_k  w dt\right)dx' \right)
    \\ & = 
    - \left( \int_{\Omega}\int_0^\infty t^{1-2s}v\p_{t}\beta_k \p_{t} w dtdx' + \int_{\Omega}a\nabla'v\cdot\nabla'\left(\int_0^\infty t^{1-2s}\beta_k  w dt\right)dx' \right).
\end{align*}
Next, we seek to argue that the second contribution vanishes by making use of the equation satisfied by $v$. To this end, we need to validate that the function $\int_0^\infty t^{1-2s}\beta_k  w dt$ is an admissible test function in this equation. To this end, we observe that 
\begin{align*}
    \left\|\int_0^\infty t^{1-2s}\beta_k  w dt\right\|_{\dot H^1(\Omega)}^2 & = \int_\Omega\left(\int_0^{2k} t^{1-2s}\beta_k  \nabla'w dt\right)^2 dx' 
    \\ & \leq 
    \int_\Omega\left(\int_0^{2k} t^{1-2s}\beta_k^2 dt\right)\left(\int_0^{2k} t^{1-2s}|\nabla'w|^2 dt\right)dx'
    \\ & \lesssim 
   \left\|w\right\|_{\dot H^1(x_{n+1}^{1-2s},\Omega\times(0,2k))}^2<\infty.
\end{align*}
Since {by the unique continuation property} $w$ vanishes on $\p\Omega\times\R_+$, we see that $\int_0^\infty t^{1-2s}\beta_k  w dt$ vanishes on $\p\Omega$, and thus $\int_0^\infty t^{1-2s}\beta_k  w dt$ belongs to $ \dot H^1_0(\Omega)$. Moreover, by the Poincar\'e inequality and the boundedness of $\Omega$ it holds that $H^1_0(\Omega)= \dot H^1_0(\Omega)$ with equivalent norms, and thus $\int_0^\infty t^{1-2s}\beta_k  w dt\in H^1_0(\Omega)\subset H^1(\Omega)$. This allows us to compute
\begin{align*}
\int_{\Omega}a\nabla'v\cdot\nabla'\left(\int_0^\infty t^{1-2s}\beta_k  w dt\right)dx' = \int_{\p\Omega}a\nabla'v\cdot\nu' \left(\int_0^\infty t^{1-2s}\beta_k  w dt\right)dx' =0,
\end{align*}
since $v\in S$ and $w=0$ in $\p\Omega\times\R_+$. Thus we are left with
\begin{align*}
    \psi(v) & = 
    -\int_{\Omega}\int_0^\infty t^{1-2s}v\p_{t}\beta_k \p_{t} w dtdx'.
\end{align*}
 Now, the desired result $\psi(v)=0$ follows by passing to the limit in $k \rightarrow \infty$. In fact, to this end, we first note that $\p_t w = - \p_t P_s u_1 + \p_t u_2$. Hence, if $\tilde w$ is any of the functions $ P_su_1, u_2$, we can estimate as in Step 1b 
\begin{align*}
     \int\limits_{k}^{k+1}  \int\limits_{\Omega} |v\p_t \tilde w| dx'dt & \leq \int\limits_{k}^{k+1} \|v\|_{L^{r_2}(\Omega)}\|\p_t \tilde w(\cdot,t)\|_{L^{r_1}(\R^n)} dt
     \\ & \leq  
     \|v\|_{L^{r_2}(\Omega)}\|\tilde w(\cdot,0)\|_{L^{q_1}(\R^n)}
     \int\limits_{k}^{k+1} t^{n(1/r_1-1/q_1)-1} dt 
     \\ & \lesssim  
     \|v\|_{L^{r_2}(\Omega)}\|\tilde w(\cdot,0)\|_{L^{q_1}(\R^n)}
      k^{n(1/r_1-1/q_1)-1} .
 \end{align*}
 As a consequence,
 \begin{align*}
     &\int\limits_{(k,k+1)\cup (R_{k,s}, R_{k,s}+1)} t^{1-2s}|\p_t\beta_k|  \int\limits_{\Omega} |v \p_t \tilde w|  dx'dt \\
     & \lesssim 
    { k^{1-2s}} \int\limits_{k}^{k+1} \int\limits_{\Omega} |v \p_t \tilde w|  dx'dt 
   +  { R_{k,s}^{1-2s}} \int\limits_{R_{k,s}}^{R_{k,s}+1} \int\limits_{\Omega} |v \p_t \tilde w|  dx'dt\\
     & \lesssim (k^{ n/r_1-n/q_1-2s} + R_{k,s}^{ n/r_1-n/q_1-2s}) \|v\|_{L^{r_2}(\Omega)}\|\tilde w(\cdot,0)\|_{L^{q_1}(\R^n)}.
 \end{align*}
 Let us choose $r_1 = \infty$, $r_2 = 1$ and $q_1=\frac{2n}{n-2}$ (for $\tilde{w}= P_s u_1)$, and $q_1=\frac{2n}{n-2s}$ (for $\tilde{w}=u_2$). We note that for both choices of $q_1$ it holds that $ n/r_1-n/q_1-2s<0$ since $n\geq 3$ and that $\|v\|_{L^1(\Omega)}\leq \|v\|_{L^2(\Omega)}<\infty$ by the bounded domain assumption. As a consequence, passing to the limit $k\rightarrow \infty$ implies that
 \begin{align*}
    \psi(v) & = 
    -\lim\limits_{k\rightarrow \infty} \int_{\Omega}\int_0^\infty t^{1-2s}v\p_{t}\beta_k \p_{t} w dtdx' = 0.
\end{align*}
 
 Eventually, we have hence obtained that if $\psi \in \tilde{H}^{-1}(\Omega)$ is such that $\psi(v_f)=0$ for all $f\in C_c^{\infty}(W)$, then also $\psi(v)=0$ for all $v\in S$. It now follows by the Hahn-Banach theorem that $V\subset H^1(\Omega)$ is dense in $S$.
\medskip

\emph{Step 3: Trace estimates.} The last step of our proof will be a trace argument. Let $g\in H^{\frac{1}{2}}(\p\Omega)$, and consider the unique solution $u\in H^1(\Omega)$ to the problem
\begin{align*}
\begin{split}
\nabla'\cdot a\nabla'u & = 0 \mbox{ in } \Omega,\\
u& = g \mbox{ in } \p\Omega.
\end{split}
\end{align*}
We have $u\in S$ by definition, and thus for all $\epsilon >0$ it is possible to find $v_\epsilon\in V$ such that $\|u-v_\epsilon\|_{H^1(\Omega)}\leq \epsilon$. Since we have the trace estimate
$$ \|g-v_\epsilon|_{\p\Omega}\|_{H^{\frac{1}{2}}(\p\Omega)} \lesssim \|u-v_\epsilon\|_{H^1(\Omega)}\leq \epsilon, $$
we see that $v_\epsilon|_{\p\Omega}\in V'$ approximates $g$ in the norm of $H^{\frac{1}{2}}(\p\Omega)$. Thus we obtain the final result $\overline {V'} =H^{\frac{1}{2}}(\partial \Omega )$.
\end{proof}

\subsection{A nonlocal proof of the density argument}
\label{sec:proof_nonlocal}

Relying on the local variant of the proof of the density result from above, we translate this into a second, completely nonlocal proof. Since in this section we follow a nonlocal approach in $\R^n$, the operator $L^s=(-\nabla' \cdot a \nabla')^s $ should be understood in a spectral or, equivalently, in a kernel representation sense. Under sufficient regularity, this interpretation is however equivalent to the previous one obtained by means of the Caffarelli-Silvestre extension, as proved in \cite{CS07,ST10}. As this only complements our result, for simplicity of presentation we only give the rigorous proof for $n\geq 5$ in which case for any $s\in (0,1)$ the problem can be treated purely with Hilbert space methods.

\begin{proof}[A nonlocal proof of Proposition \ref{prop:density}]
Let $n\geq 5$.
We again argue by the Hahn-Banach theorem. Using the notation from Section \ref{sec:proof_local}, it suffices to prove that if for some $\psi \in \tilde{H}^{-1}(\Omega)$ it holds that $\psi(v_f)=0$ for all $f\in C_c^{\infty}(W)$, then also $\psi(v)=0$ for all $v\in S$.

To this end, we fix $\psi\in \tilde H^{-1}(\Omega)$ satisfying the Hahn-Banach assumption, and then define $w_1 \in \dot{H}^1(\R^n)\cap L^{\frac{2n}{n-2}}(\R^n)$ to be the weak solution to $L w_1 = \psi$ in $\R^n$, i.e., we assume that
 \begin{align*}
 \int\limits_{\R^n} a\nabla w_1 \cdot \nabla \varphi dx = \psi (\varphi)  \qquad\mbox{ for all } \varphi \in \dot{H}^1(\R^n).
 \end{align*}
 We observe that since $\psi \in \tilde{H}^{-1}(\Omega)$ is compactly supported and since, by assumption, $n\geq 5$, the decay of the fundamental solution implies that $w_1 \in H^1(\R^n)$. In particular, we also obtain that $w_1 \in H^s(\R^n)$ for any $s\in (0,1)$.
 Further, we let $u_1\in  H^s(\R^n)$ be the unique solution of 
	\begin{align*}
		L^su_1 &=0\mbox{ in }\Omega,\\
		u_1 &= w_1      \mbox{ in }\Omega_e,
	\end{align*}
which is a well-posed problem in light of the fact that $w_1\in H^{s}(\R^n)$. 

\emph{Step 1:}
With the above notation and the decay and regularity properties, we obtain
\begin{align}
\label{eq:Hahn_Banach1}
0 &= \psi(v_f) = \langle  v_f, \psi \rangle_{H^{1}(\Omega),\tilde{H}^{-1}(\Omega)} 
= \langle  v_f, \psi \rangle_{H^{1}(\R^n), H^{-1}(\R^n)} =
\int\limits_{\R^n} a\nabla v_f  \cdot \nabla w_1  dx.
\end{align}
Next, we seek to show that 
\begin{align*}
\int\limits_{\R^n} a\nabla v_f  \cdot \nabla w_1  dx = (L^{s/2}u_f, L^{s/2} w_1)_{L^2(\R^n)}.
\end{align*}
To this end, we use two approximation arguments: 
Firstly, we observe that $Lv_f = L^su_f$ holds in a weak sense in $\R^n$ by Lemma \ref{lemma1} and a density argument. Indeed, if $\varphi\in C^\infty_c(\R^n)$ and $\{u_{f,k}\}_{k}\subset C^\infty_c(\R^n)$ is a sequence of smooth functions such that $\|u_{f,k}-u_f\|_{H^s(\R^n)}\leq 1/k$ for all $k\in\N$, then by the definition of $v_{f,k}:=L^{s-1}u_{f,k}$ we have $$(a\nabla v_{f,k},\nabla\varphi)_{L^2(\R^n)} = (L^{s/2}u_{f,k},L^{s/2}\varphi)_{L^2(\R^n)},$$
with
$$ \|\nabla(v_f-v_{f,k})\|_{L^2(\R^n)} \lesssim \|u_f-u_{f,k}\|_{H^s(\R^n)} \leq 1/k $$
and
$$ \|L^{s/2}(u_f-u_{f,k})\|_{L^2(\R^n)} \lesssim \|u_f-u_{f,k}\|_{H^s(\R^n)} \leq 1/k. $$
Secondly, by a similar approximation argument we have 
$$( L^{s/2} u_f,L^{s/2}w_{1,k} )_{L^2(\R^n)} = (a\nabla v_f, \nabla w_{1,k})_{L^2(\R^n)},$$
with
\begin{align*}
 \|\nabla(w_1-w_{1,k})\|_{L^2(\R^n)} &\leq \|w_1-w_{1,k}\|_{H^1(\R^n)} \leq 1/k, \\
 \|L^{s/2}(w_1-w_{1,k})\|_{L^2(\R^n)} &\leq \|w_1-w_{1,k}\|_{H^s(\R^n)} \leq \|w_1-w_{1,k}\|_{H^1(\R^n)} \leq 1/k. 
 \end{align*}
Thus, by \eqref{eq:Hahn_Banach1} and passing to the limit in the above two approximation arguments, we obtain
\begin{align}
\label{eq:Hahn_Banach2}
0=(a\nabla v_f, \nabla w_1)_{L^2(\R^n)}=( L^{s/2} u_f,L^{s/2}w_1 )_{L^2(\R^n)}.
\end{align}
Moreover, since $w_1-u_1$ and $u_f -f$ both belong to $\tilde H^s(\Omega)$, using the weak form of the equations satisfied by $u_f$ and $u_1$ we have
\begin{align}
\label{eq:Hahn_Banach3}
\begin{split}
    (  L^{s/2} u_f,   L^{s/2}w_1  )_{L^2(\R^n)} & - (  L^{s/2} f,   L^{s/2}u_1  )_{L^2(\R^n)}  \\ 
    & = (  L^{s/2} u_f,   L^{s/2}(w_1-u_1)  )_{L^2(\R^n)} + ( L^{s/2} (u_f-f),   L^{s/2} u_1  )_{L^2(\R^n)} = 0 .
    \end{split}
\end{align}
Therefore, combining \eqref{eq:Hahn_Banach2}, \eqref{eq:Hahn_Banach3}, we infer that
\begin{align*}
\langle   f,   L^s u_1  \rangle_{H^s(\R^n),H^{-s}(\R^n)} = (  L^{s/2} f,   L^{s/2}u_1  )_{L^2(\R^n)}=0.
\end{align*}
Since this holds for all $f\in C_c^{\infty}(W)$, we obtain that $ L^s u_1 = 0$ in $W$. As $u_1 = w_1$ in $W$ and since $L$ is a local operator, by recalling the equation for $w_1$ we also have $L u_1 = 0$ in $W$. As a consequence, we infer that for $v_1:= L^s u_1 \in H^{-s}(\R^n)$ it holds that
\begin{align*}
v_1 = 0 , \ L^{1-s} v_1 = 0 \mbox{ in } W.
\end{align*}
By the global unique continuation property for the fractional  elliptic operator $L^{1-s}$  \cite{R15,GSU20,GRSU20}, we hence deduce that $v_1 = 0$ in $\R^n$. As a consequence, $L^s u_1 = 0$ in $\R^n$. By virtue of the fact that $u_1 \in H^s(\R^n)$, we therefore conclude that $u_1 = 0$ in $\R^n$. As $u_1 = w_1$ in $\Omega_e$, this also implies that $w_1 = 0$ in $\Omega_e$.\\

\emph{Step 2:} We conclude the density argument similarly as above by noting that for $v \in S$ we have
\begin{align*}
\psi(v) = -\int\limits_{\Omega}  a \nabla' w_1 \cdot \nabla' v dx
= \langle w_1,   \nu \cdot a \nabla' v \rangle_{\partial \Omega}-\langle \nu \cdot a \nabla' w_1, v \rangle_{\partial \Omega}  = 0.
\end{align*}
Here we first noted that $w_1 = 0$ in $\Omega_e$ and then, in the last step, we used the fact that on $\partial \Omega$ it holds (in an $H^{\frac{1}{2}}(\partial \Omega)$ and $H^{-\frac{1}{2}}(\partial \Omega)$ sense, respectively) $w_1 = 0 = \nu \cdot a \nabla' w_1$ by Step 1.
\end{proof}

\subsection{Proof of Theorem \ref{thm:main}}

With the density result of Proposition \ref{prop:density} in hand, we turn to the proofs of Theorem \ref{thm:main} and Proposition \ref{prop:density_inversion}:

\begin{proof}[Proof of Theorem \ref{thm:main} and Proposition \ref{prop:density_inversion}]
By Proposition \ref{prop:density}, the operator 
\begin{align*}
{T_1}: \tilde{H}^s(W) \ni f \mapsto v_f(x')|_{\partial \Omega}:=\int\limits_{0}^{\infty} t^{1-2s} u_f(x',t)dt|_{\partial \Omega} \in H^{\frac{1}{2}}(\partial \Omega)
\end{align*}
is linear, bounded and has a dense image. 
Due to the continuity of the map $\Lambda: H^{\frac{1}{2}}(\partial \Omega) \rightarrow H^{-\frac{1}{2}}(\partial \Omega)$, this then implies that
\begin{align*}
\overline{T(\mathcal{C}_{s,a})}^{H^{\frac{1}{2}}(\partial \Omega) \times H^{-\frac{1}{2}}(\partial \Omega)} = \mathcal{C}_{a} 
\end{align*}
as claimed in Proposition \ref{prop:density_inversion}.
Next, by unique continuation, $(f|_{W},(-\D)^s u_f|_{W})$ determines $(x_{n+1}^{1-2s} u_f |_{\partial \Omega \times \R_+}, x_{n+1}^{1-2s}\p_{\nu}  u_f |_{\partial \Omega \times \R_+})$ and thus also $(v_f|_{\partial \Omega}, \p_{\nu} v_f|_{\partial \Omega})$. Combined with the density result, this proves that $\Lambda_s$ determines $\Lambda$.
\end{proof}

\section{Proof of Theorem \ref{thm:consequences}}
\label{sec:cons}
In this section, for completeness, we outline how the results of Theorem \ref{thm:main} imply the ones from Theorem \ref{thm:consequences}. This follows without any changes along the same lines as the argument from \cite{GU21} and is only included for the convenience of the reader.

\begin{proof}
Let $a_1,a_2 \in \R^{n\times n}$ be as in the statement of the theorem. Consider the fractional and classical Calder\'on problems having $a_j$ as coefficients, i.e.,
\begin{align*}
\begin{split}
(\nabla \cdot a_j \nabla)^s u & = 0 \mbox{ in } \Omega,\\
u & = f \mbox{ on } \Omega_e,
\end{split}
\end{align*}
and
\begin{align*}
\begin{split}
\nabla \cdot a_j \nabla v & = 0 \mbox{ in } \Omega,\\
v & = g \mbox{ on } \partial \Omega,
\end{split}
\end{align*}
for $j=1,2$. Here the Dirichlet data $f,g$ are such that $f\in H^s(\R^n)$ and $g\in H^{\frac{1}{2}}(\p\Omega)$, while the unique solutions $u,v$ are such that $u\in H^s(\R^n)$ and $v\in H^{1}(\Omega)$.   Assume that the nonlocal Dirichlet-to-Neumann maps {$\Lambda_{s,1}, \Lambda_{s,2}$}, which correspond to $a_1,a_2$ respectively, coincide. By Theorem \ref{thm:main}, this implies that the local Dirichlet-to-Neumann maps $\Lambda_{1}, \Lambda_{2}$, which correspond to $ a_1, a_2$ respectively, themselves coincide, since they are uniquely determined by $\Lambda_{s,1}, \Lambda_{s,2}$, respectively. Since by assumption the local Dirichlet-to-Neumann map $\Lambda_j$ uniquely determines the coefficient $a_j$ in the classical Calder\'on problem for $j=1,2$, we deduce that $a_1=a_2$ holds. This proves the theorem.
\end{proof}

\section{Tikhonov regularization}
\label{sec:tikh}

In this section we outline that the procedure of reconstructing the Dirichlet-to-Neumann data for the local problem from the nonlocal data is completely constructive. To this end we present the argument for Proposition \ref{prop:Tikhonov}.

\begin{proof}
The proof is based on the results contained in \cite[Chapter 4]{CK}. We start by observing that the linear operator $A$ is injective by unique continuation. In fact, if $\tilde u\in \mathcal V$ is such that $A\tilde u = 0$, then $\tilde u(x',0) = \lim\limits_{t\rightarrow 0}t^{1-2s}\p_{t}\tilde u(x',t) = 0$ on $W$, which implies $\tilde u \equiv 0$ on $\Omega_e\times \R_+$. As a result of the mapping property $A: \mathcal{V} \mapsto H^s(W)\times H^{-s}(W)$, the compact inclusions $H^s(W) \hookrightarrow H^{s-\epsilon}(W)$ and $H^{-s}(W) \hookrightarrow H^{-s-\epsilon}(W)$ as well as the injectivity of $A$ and \cite[Theorem 4.13]{CK}, we have that the operator $R_\alpha := (\alpha I +A^*A)^{-1}A^*$ is well-defined and a regularization scheme for $A$, which means that
$$\lim\limits_{\alpha\rightarrow 0} \|R_\alpha A \varphi - \varphi\|_{\dot H^1(x_{n+1}^{1-2s},\Omega_e\times\R_+)}=0, \qquad \mbox{ for all } \varphi\in \mathcal V.$$
 Moreover, since $H^{s}(W) \times H^{-s}(W)$ is compactly embedded in  $H^{s-\epsilon}(W) \times H^{-s-\epsilon}(W)$ for any $\epsilon>0$, we deduce that $A$ is compact. Using this and \cite[Theorem 4.14]{CK}, we obtain that $\tilde u_\alpha := R_\alpha ((f,\Lambda_sf))$ is the unique minimizer in $\mathcal V$ of the energy functional $J_\alpha$ corresponding to $(f,\Lambda_sf)\in  H^{s}(W) \times H^{-s}(W)$. 
\end{proof}

\section{Regularity estimates}
\label{sec:reg}

In this section we prove that the function $v(x'):=\int_0^\infty t^{1-2s}\tilde u(x',t)dt$, which was introduced in \eqref{eq:v} and was extensively used in the arguments above, satisfies the following regularity properties which were used in our arguments in the previous sections.

\begin{prop}
\label{prop:reg}
Let $\Omega, W \subset \R^n$ be bounded, open Lipschitz sets with $\overline{\Omega}\cap \overline{W} = \emptyset$.
Let $n\geq 3$, $u \in H^{s}(\R^n)$ with compact support in $\overline{W}\cup \overline{\Omega}$ and $\tilde{u}\in H^1(\R^{n+1}_+, x_{n+1}^{1-2s})$ be the associated Caffarelli-Silvestre extension with coefficient matrix $\tilde{a}$ as in the introduction. 
Let $w: {\R^{n+1}_+} \rightarrow \R$ be given by 
\begin{align*}
w(x',x_{n+1}):= \int\limits_{x_{n+1}}^{\infty} t^{1-2s} \tilde{u}(x',t) dt.
\end{align*}
Then, for $x_{n+1}>0$ the function $w$ is bounded and the limit $w(x,0) $ exists with $w(\cdot,0)\in H^1_{loc}(\R^n)$.
\end{prop}

Before turning to the proof of {Proposition \ref{prop:reg}}, we show some decay estimates in the vertical direction for the Caffarelli-Silvestre extension. In order to simplify the notation, in what follows we will frequently make use of the abbreviation $L:=\nabla'\cdot a\nabla'$.

\begin{lem}\label{CS-estimates-vertical}
Let $n\in \N$ and $u \in H^{s}(\R^n)$, and denote by $\tilde{u}\in \dot H^1(\R^{n+1}_+, x_{n+1}^{1-2s})$ its associated Caffarelli-Silvestre extension with coefficient matrix $\tilde{a}$ as in the introduction. Then, for $y>0$ the function $\tilde{u}$ satisfies the following bounds
\begin{equation}\label{eq:CS-decay}
    |\tilde{u}(x',y)| \lesssim y^{-n}  \|u\|_{L^1(\R^n)}, \qquad |\nabla'\tilde{u}(x',y)| \lesssim y^{-n-1}  \|u\|_{L^1(\R^n)}.
\end{equation}
Moreover, if $1\leq r,p,q \leq \infty$ are such that $1+ \frac{1}{r} = \frac{1}{p} + \frac{1}{q}$, then $\tilde u$ also satisfies the estimates 
\begin{equation}\label{eq:extension_est}
\|\tilde u(\cdot, y)\|_{L^r(\R^n)} \lesssim y^{n/p-n}\|u\|_{L^{q}(\R^n)},\qquad \|(\nabla \tilde u)(\cdot, y)\|_{L^r(\R^n)} \lesssim y^{n/p-n  - 1}\|u\|_{L^{q}(\R^n)}.
\end{equation}
\end{lem}

\begin{proof}
Throughout the proof we will consider the situation that $y>0$.
    In order to infer the desired results, we use the heat kernel representation of the Poisson formula for the Caffarelli-Silvestre extension \cite[Theorem 2.1]{ST10}
\begin{equation}
\label{eq:CS_Stinga}
\tilde{u}(x',y) = c_s y^{2s}  \int\limits_{\R^n} \int\limits_{0}^{\infty} K_t(x',z) e^{-\frac{y^2}{4t}} \frac{dt}{t^{1+s}} u(z) dz.
\end{equation}
Here $K_t$ denotes the heat kernel, which for $t>0$ verifies $\p_tK_t(x,z) = LK_t(x,z)$ and the following estimates (see \cite{ST10} for the case $k=0$, and \cite{CJKS20} for the case $k=1$):
\begin{equation}\label{estimate-kernel}
|D_{x'}^kK_t(x',z)|\lesssim t^{-\frac{n+k}{2}} e^{-c \frac{|x'-z|^2}{t}}, \qquad k\in\{0,1\}.
\end{equation}
Using that
\begin{align}\label{formula-integral}
\int_0^\infty \frac{e^{-A/4t}}{t^{n/2+\sigma+1}}dt \approx A^{-n/2-\sigma},
\end{align}
(see, for instance, the argument on p.2104 in \cite{ST10}),
we compute {for $k\in \{0,1\}$}
\begin{equation*}\begin{split}
    |D_{x'}^k\tilde{u}(x',y)| & \lesssim y^{2s}  \int\limits_{\R^n} \int\limits_{0}^{\infty}|D_{x'}^k K_t(x',z)| e^{-\frac{y^2}{4t}} \frac{dt}{t^{1+s}} |u(z)| dz
    \\ & 
    \lesssim y^{2s}  \int\limits_{\R^n} |u(z)| \int\limits_{0}^{\infty}t^{-(\frac{n+k}{2}+1+s)} e^{-(c \frac{|x'-z|^2}{t}+\frac{y^2}{4t})} dt  dz
    \\ & 
    \lesssim y^{2s}  \int\limits_{\R^n} \frac{|u(z)|}{(|x'-z|^2+y^2)^{\frac{n}{2}+s+\frac{k}{2}}} dz
    \\ &
    = (|u|\ast K_{k/2,y})(x'),
\end{split}\end{equation*}
where $K_{\rho,y}(x'):= \frac{y^{2s}}{(|x'|^2 + y^{2})^{\frac{n}{2} +s + \rho}}$. In particular, {by Young's convolution inequality,} this implies the estimate
\begin{equation*}
    |D_{x'}^k\tilde{u}(x',y)| \lesssim y^{-n-k}  \|u\|_{L^1(\R^n)}, \ { k \in\{0,1\} } .
\end{equation*}
Similarly, we can compute
\begin{equation*}\begin{split}
    |\p_{y}\tilde{u}(x',y)| & \lesssim y^{2s-1}  \int\limits_{\R^n} \int\limits_{0}^{\infty}|K_t(x',z)| e^{-\frac{y^2}{4t}} \frac{dt}{t^{1+s}} |u(z)| dz + y^{2s+1}  \int\limits_{\R^n} \int\limits_{0}^{\infty}|K_t(x',z)| e^{-\frac{y^2}{4t}} \frac{dt}{t^{2+s}} |u(z)| dz
    \\ & \lesssim 
    y^{-1}  (|u|\ast K_{0,y})(x') + y (|u|\ast K_{1,y})(x').
\end{split}\end{equation*}
Therefore, by Young's convolution inequality we {deduce that}
\begin{align*}
\|\tilde u(\cdot, y)\|_{L^r(\R^n)} &\lesssim \|K_{0,y}(\cdot)\|_{L^p(\R^n)}\|u\|_{L^{q}(\R^n)},
\\ \|(\nabla \tilde u)(\cdot, y)\|_{L^r(\R^n)} &\lesssim \left(\|K_{1/2,y}(\cdot)\|_{L^p(\R^n)} + y^{-1}\|K_{0,y}(\cdot)\|_{L^p(\R^n)} + y\|K_{1,y}(\cdot)\|_{L^p(\R^n)} \right)\|u\|_{L^{q}(\R^n)},
\end{align*}
where $1+ \frac{1}{r} = \frac{1}{p} + \frac{1}{q}$. By changing variables, we obtain
\begin{align*}
\|K_{\rho,y}(\cdot)\|^p_{L^p(\R^n)} & \approx \int\limits_{0}^\infty \frac{y^{2sp}r^{n-1}}{(r^2 + y^{2})^{p(\frac{n}{2} +s + \rho)}}dr = \int\limits_{0}^\infty \frac{y^{n-pn  - 2\rho p}\tau^{n-1}}{(\tau^2 + 1)^{p(\frac{n}{2} +s + \rho)}}d\tau  \lesssim y^{n-pn  - 2\rho p},
\end{align*}
and thus
\begin{equation*}
\|\tilde u(\cdot, y)\|_{L^r(\R^n)} \lesssim y^{n/p-n}\|u\|_{L^{q}(\R^n)},\qquad \|(\nabla \tilde u)(\cdot, y)\|_{L^r(\R^n)} \lesssim y^{n/p-n  - 1}\|u\|_{L^{q}(\R^n)}.
\end{equation*}
\end{proof}

{With this auxiliary result in hand, we now turn to the proof of Proposition \ref{prop:reg}.}

\begin{proof}[Proof of Proposition \ref{prop:reg}] 
{In deducing the desired result, we argue in three steps.}

\emph{Step 1: {Well-definedness of the function $w$ for $x_{n+1}>0$, i.e., convergence of the integral defining it.}} Inequality \eqref{eq:CS-decay} immediately implies the following estimate for $w$ {and $y>0$}:
\begin{equation}\begin{split}\label{eq:function-w-decay}
    |D_{x'}^kw(x',y)| & \leq  \int\limits_{y}^{\infty} t^{1-2s} |D_{x'}^k\tilde{u}(x',t)| dt \lesssim y^{2-2s-n-k}  \|u\|_{L^1(\R^n)}, \qquad k\in\{0,1\} .
\end{split}\end{equation}
Given that $\supp(u)$ is compact and $u\in L^2(\R^n)$, we have $u\in L^1(\R^n)$, and thus the above expression is finite for a.e. $x=(x',x_{n+1})\in \R^{n+1}_+$. In particular, $w$ is well-defined. \\

\emph{Step 2: {The local $L^2$ estimate for $v(x'):=w(x',0)$.}}
We further seek to prove that $v(x'):=w(x',0)\in L^2(\Omega')$ for any $\Omega'\subset \R^n$ bounded and open. To this end, we aim at upgrading the right hand side of estimate \eqref{eq:function-w-decay} to be independent of the vertical variable. By virtue of estimate \eqref{eq:function-w-decay}, we may apply Fubini's theorem and infer for $k\in\{0,1\}$ and $y>0$
\begin{equation}\begin{split}\label{eq:better-estimate}
    |D_{x'}^kw(x',y)| & \lesssim \int\limits_{y}^{\infty}  \int\limits_{\R^n} \frac{t|u(z)|}{(|x'-z|^2+t^2)^{\frac{n+k}{2}+s}} dz dt
    \\ &
    \leq \int\limits_{\R^n} |u(x'+z)|\int\limits_{0}^{\infty}   \frac{t}{(|z|^2+t^2)^{\frac{n+k}{2}+s}} dt dz
    \\ &
    \approx \int\limits_{B_R} \frac{|u(x'+z)|}{|z|^{n+k+2s-2}} dz,
\end{split}\end{equation}
for some $R>0$ large enough, depending on the support of $u$. 
 We note that all the estimates from above are independent of $y \in (0,\infty)$, hence, the estimate transfers to the limit $v$ of $w(\cdot,y)$ as $y\rightarrow 0$.
Since $|\cdot|^{-(n+2s-2)}$ is locally integrable, $\supp(u)\subset \overline{\Omega} \cup \overline{W}$, and $\Omega'$ is bounded, for some large $r>0$ we obtain 
\begin{align*}
\|v(\cdot)\|_{L^2(\Omega')}^2 
&\leq C_{s} \int\limits_{\Omega'} \left| \int\limits_{\R^n} \frac{|u(z)|}{|x'-z|^{n+2s-2}}dz \right|^2 dx' \\
&\leq C_{s} \int\limits_{\R^n} \left| \int\limits_{\R^n} \frac{|u(z)|\chi_{B_r}(z-x')}{|z-x'|^{n+2s-2}}dz \right|^2 dx' \\ 
& =  C_{s} \left\| u\ast (\chi_{B_r}|\cdot|^{-n-2s+2}) \right\|_{L^2(\R^n)}^2 \\
& \leq C_s \|\chi_{B_r}(\cdot) |\cdot|^{-n-2s+2}\|_{L^1(\R^n)}^2 \|u\|_{L^2(\R^n)}^2\\
& \leq C_{n,s} \|u\|_{L^2(\R^n)}^2.
\end{align*}
Here we have used Young's convolution inequality in the last line together with the integrability of $\chi_{B_r}(\cdot) |\cdot|^{-n-2s+2}$. 
\\

\emph{Step 3: {The gradient estimate for $v(x')$, $x' \in \Omega'$.}} It remains to argue that we have $v\in H^1(\Omega')$.  {For $s\in (0,\frac{1}{2})$ this essentially follows from the argument in Step 2. For $s \in [\frac{1}{2},1)$, the analysis in the previous step is not detailed enough (as the kernel then is not integrable at zero). We thus need to perform a more accurate analysis.} 
To this end, let $0<r<R$ with $\supp(u)\subseteq B_r$, and assume that $g\in C^\infty_c(B_R)$ takes values in $[0,1]$ and verifies $g=1$ in $B_r$. {We denote the Caffarelli-Silvestre extension of $g$ by $\tilde{g}$.} Then, as in the proof of Lemma \ref{CS-estimates-vertical} and as in formula \eqref{eq:better-estimate}, we compute for $x_{n+1}>0$
\begin{align*}
|\nabla_{x'}w(x',x_{n+1})&-\nabla_{x'}\left[ \int_{x_{n+1}}^\infty y^{1-2s}\tilde g(x',y)dy \right]u(x')| = \\ & = c_s\left|\int_{x_{n+1}}^\infty y\int_{\R^n}\int_0^\infty \nabla_{x'}K_t(x',z)e^{-\frac{y^2}{4t}} \frac{dt}{t^{1+s}}(u(z)-g(z)u(x'))dzdy\right|
\\ & \lesssim \int\limits_{x_{n+1}}^{\infty}  \int\limits_{\R^n} \frac{y|u(z)-g(z)u(x')|}{(|x'-z|^2+y^2)^{\frac{n+1}{2}+s}} dz dy
    \\ & \lesssim \int_{\R^n}\frac{|u(z)-g(z)u(x')|}{|x'-z|^{n+2s-1}} dz. 
\end{align*}
Using the support condition of $g$ {and $u$}, we now obtain
\begin{align*} |\nabla_{x'}w(x',x_{n+1})-\nabla_{x'}\left[ \int_{x_{n+1}}^\infty y^{1-2s}\tilde g(x',y)dy \right]u(x')| & \lesssim \int_{B_R}\frac{|u(z)-u(x')|}{|x'-z|^{n+2s-1}} dz, 
\end{align*}
and therefore
\begin{align*}
\|\nabla_{x'}w(x',0)-\nabla_{x'}&\left[ \int_{0}^\infty y^{1-2s}\tilde g(x',y)dy \right]u(x')\|_{L^2(\Omega')}^2 \lesssim \int_{\Omega'} \left( \int_{B_R} \frac{|u(z)-u(x')|}{|x'-z|^{n+2s-1}} dz \right)^2 dx' \\ & \leq \int_{\Omega'} \left(\int_{B_R} \frac{|u(z)-u(x')|^2}{|x'-z|^{n+2s}} dz\right) \left( \int_{B_R} \frac{1}{|x'-z|^{n+2s-2}} dz\right) dx' \\
& \lesssim \int_{\Omega'} \int_{B_R} \frac{|u(z)-u(x')|^2}{|x'-z|^{n+2s}} dz dx' \\
& \leq \|u\|_{H^s(\R^n)}^2.
\end{align*}
Here we used that $H^s(\R^n) = W^{s,2}(\R^n)$ (see \cite[Theorem 3.16]{McLean}).

Since we are eventually interested in $\|\nabla_{x'} w(x',0)\|_{L^2(\Omega')}$, and as also
$$ \left\|\nabla_{x'}\left[ \int_{0}^\infty y^{1-2s}\tilde g(x',y)dy \right]u(x')\right\|_{L^2(\Omega')} \leq \|u\|_{L^2(\Omega')}\left\|\nabla_{x'} \int_{0}^\infty y^{1-2s}\tilde g(x',y)dy\right\|_{L^\infty(\Omega')},  $$
we {next seek} to show that 
\begin{align*}
|\nabla_{x'} \int_{0}^\infty y^{1-2s}\tilde g(x',y)dy|<{C<\infty \mbox{ for all } x' \in \overline{\Omega'}.}
\end{align*}
In other words, we aim to show that
$$ \left| \nabla_{x'}\int_{0}^\infty  \int\limits_{\R^n} \int\limits_{0}^{\infty} K_t(x',z) ye^{-\frac{y^2}{4t}} \frac{dt}{t^{1+s}} g(z) dzdy \right|<{C<}\infty, $$ or equivalently,
\begin{align}
\label{eq:idggrad}
\left| \nabla_{x'}\int_{B_R}g(z)\int_{0}^\infty  t^{-s}  K_t(x',z) dtdz \right|{<C<}\infty,
\end{align}
uniformly in $x'\in \Omega'$.

In order to obtain such bounds, we first study the auxiliary function $\int_{B_R}g(z)\int_{0}^\infty  t^{-s}  K_t(x',z) dtdz$.
Here we observe that $t^{-s} = \frac{\p_t(t^{1-s})}{1-s}$. Hence, integrating by parts in $t$, we obtain
\begin{align}
\label{eq:idg}
\begin{split}
-\int_{B_R}g(z)\int_{\varepsilon}^\infty  t^{-s}  K_t(x',z) &dtdz  \approx- \int_{B_R}g(z)\int_{\varepsilon}^\infty  \p_t(t^{1-s})  K_t(x',z) dtdz \\ & = \varepsilon^{1-s}\int_{ B_R}g(z)K_\varepsilon(x',z)dz - \int_{ B_R}g(z)\lim\limits_{t\rightarrow\infty}t^{1-s}K_t(x',z)dz   \\ & \quad + \int_{B_R}g(z)\int_{\varepsilon}^\infty  t^{1-s}  \p_tK_t(x',z) dtdz 
\\ & = \varepsilon^{1-s}\int_{ B_R}g(z)K_\varepsilon(x',z)dz + \int_{B_R}g(z)\int_{\varepsilon}^\infty  t^{1-s}  \p_tK_t(x',z) dtdz,
\end{split}
\end{align}
where the second boundary term vanishes in light of the fact that for $n\geq 3$ we have
$$ |\lim\limits_{t\rightarrow\infty}  t^{1-s}  K_t(x',z)| \lesssim \lim\limits_{t\rightarrow\infty}  t^{1-s-\frac{n}{2}} e^{-c \frac{|x'-z|^2}{t}} =0.  $$
We next focus on the second contribution on the right hand side of \eqref{eq:idg}:
By the properties of the heat kernel, a double integration by parts in $z$ gives
\begin{align*} \int_{B_R}g(z)\int_{\varepsilon}^\infty  t^{1-s}  \p_tK_t(x',z) dtdz& = \int_{B_R}g(z)\int_{\varepsilon}^\infty  t^{1-s}  L_zK_t(x',z) dtdz \\ & = \int_{B_R}g(z) L_z \int_{\varepsilon}^\infty  t^{1-s}  K_t(x',z) dtdz \\  & = \int_{B_R\setminus B_r} L g(z) \int_{\varepsilon}^\infty  t^{1-s} K_t(x',z) dtdz.
\end{align*}
Here we note that the dominated convergence theorem is applicable, since by \cite[equation (0.6)]{Gri95}  we have the bound
$|\partial_tK_t(x',z)|\lesssim \frac{1}{t^{n/2+1}}e^{ -c \frac{|x'-z|^2}{t}}$, and thus
\begin{equation*}
    t^{1-s}| L_zK_t(x',z)| \lesssim  \frac{1}{t^{n/2+s}} e^{ -c \frac{|x'-z|^2}{t}} \leq \frac{1}{t^{n/2+s}},
\end{equation*}
which is integrable for $t>\varepsilon$.
Returning the the first expression in \eqref{eq:idg}, we note that since for $\epsilon \rightarrow 0$
\begin{align*}
\int_{ B_R}g(z)K_\varepsilon(x',z)dz \rightarrow g(x') \mbox{ for } x'\in \R^n,
\end{align*}
we obtain that $\varepsilon^{1-s}\int_{ B_R}g(z)K_\varepsilon(x',z)dz \rightarrow 0$ for $\epsilon \rightarrow 0$. We further observe that, for almost every $z\in B_R$, we have
$$ | L g(z)|\left| \int_{\varepsilon}^\infty t^{1-s}K_t(x',z)dt\right| \lesssim | L g(z)| \int_{\varepsilon}^\infty t^{1-s-n/2}e^{ -c \frac{|x'-z|^2}{t}}  dt \lesssim \frac{| L g(z)|}{|x'-z|^{n+2s-4}},$$
$$ g(z)\left| \int_{\varepsilon}^\infty t^{-s}K_t(x',z)dt\right| \lesssim g(z) \int_{\varepsilon}^\infty t^{-s-n/2}e^{ -c \frac{|x'-z|^2}{t}}  dt \lesssim \frac{ g(z)}{|x'-z|^{n+2s-2}},$$
where the right hand sides are all integrable functions of $z$ due to the support condition on $g$.
Therefore, {returning to the expression in \eqref{eq:idg},} by dominated convergence we deduce
$$ \int_{B_R}g(z)\int_{0}^\infty  t^{-s}  K_t(x',z) dtdz  \approx  \int_{B_R\setminus B_r}L g(z) \int_{0}^\infty  t^{1-s} K_t(x',z) dtdz.$$

With this in hand, we study the tangential gradient of the above expression and thus bound the quantity from \eqref{eq:idggrad}.
To this end, we note that there exists a constant $\delta>0$ such that $|x'-z|>\delta$ for all $z\in B_R\setminus B_r$. Thus, 
$$ \left| L g(z) t^{1-s}\nabla_{x'}K_t(x',z) \right| \lesssim | L g(z)| t^{-s-n/2}|x'-z|e^{ -c \frac{|x'-z|^2}{t}} \lesssim | L g(z)| t^{-s-n/2}e^{-\frac{c\delta^2}{t}}, $$
with $$\int_{B_R\setminus B_r}\int_0^\infty |Lg(z)|t^{-s-n/2}e^{-\frac{c\delta^2}{t}}dtdz \lesssim \int_{B_R\setminus B_r} |Lg(z)|dz<\infty.$$
As a consequence, it is clear that
\begin{align*}
    \left| \nabla_{x'}\int_{\R^n}g(z)\int_{0}^\infty  t^{-s}  K_t(x',z) dtdz \right| & \lesssim \int_{B_R\setminus B_r}| L g(z)|\int_{0}^\infty t^{1-s} |\nabla_{x'}K_t(x',z)|dtdz < \infty.
\end{align*}
\end{proof}

In the next short lemma we further observe that $v_f:= L^{s-1} u_f \in \dot{H}^1(\R^n)$ if $u_f \in H^s(\R^n)$ is compactly supported.

\begin{lem}\label{lemma1}
Let $u_f \in H^s(\R^n)$ with $\supp(u_f) \subset \R^n$ compact. Then, $v_f:= L^{s-1} u_f \in \dot{H}^1(\R^n)$ and $\|\nabla v_f\|_{L^2(\R^n)} \leq C \|u_f\|_{H^s(\R^n)}$.
In particular, if
$u_{f,k} \rightarrow u_f$ in $H^s(\R^n)$, then $v_{f,k}:= L^{s-1}u_{f,k}\rightarrow v_f := L^{s-1}u_f $ in $\dot H^1(\R^n)$.
\end{lem}

\begin{proof}
Let $R\in\R$ be such that $\supp(u_f)\subseteq B_R$. By the argument in Proposition \ref{prop:reg} we obtain the local estimate $\|\nabla v_f\|_{L^2(B_{2R})}\lesssim \|u_f\|_{H^s(\R^n)}$. We then
use the compact support of $u_f$ and Lemma \ref{CS-estimates-vertical} to compute
\begin{align*}
    \|\nabla v_f\|_{L^2(\R^n \setminus B_{2R})}^2 & \lesssim \int_{\R^n \setminus B_{2R}} \left|\int_{B_R}\frac{|u_f(z)|}{|x'-z|^{n+2s-1}}dz\right|^2dx' 
    \\ & \lesssim
    \int_{\R^n \setminus B_{2R}} \left|\int_{B_R}\frac{|u_f(z)|}{|x'|^{n+2s-1}}dz\right|^2dx'
    \\ & \lesssim \|u_f\|_{L^2(\R^n)}^2,
\end{align*}  
where we used that $n+4s >2$.
Together with the argument in Proposition \ref{prop:reg} this proves the desired global estimate.
\end{proof}

Finally, we prove one more auxiliary Lemma, which was used in the proof of Lemma \ref{well-posed-adjoint}.

\begin{lem}\label{CS-Lp}
    Let $n\geq 3$, $s\in (0,1)$, $q\in(1,\infty)$, $R>0$, and assume that $\Omega\Subset B_R \subset \R^n$ is a bounded open set. If $u\in L^{\frac{2n}{n-2}}(\R^n)$ is supported in the complement of $B_{R}$, then 
    $$\lim\limits_{x_{n+1}\rightarrow 0} x_{n+1}^{1-2s}\p_{n+1}\tilde u \in L^2(\Omega),$$
    where $\tilde u$ is the Caffarelli-Silvestre extension of $u$ corresponding to the operator $L:=\nabla'\cdot a\nabla'$.  
\end{lem}

\begin{proof}
    Let $q:= \frac{2n}{n-2}$. We start by observing that since $u\in L^q(\R^n)$ and \cite[Theorem 2.1]{ST10} holds, the extension $\tilde u$ is well-defined. As in the proof of Lemma \ref{CS-estimates-vertical} we get
\begin{equation*}
    |\p_{y}\tilde{u}(x',y)| \lesssim 
    y^{-1}  (|u|\ast K_{0,y})(x') + y (|u|\ast K_{1,y})(x'),\end{equation*}
    and therefore by H\"older's inequality
\begin{align*}
    \|\lim\limits_{y\rightarrow 0}&y^{1-2s}\p_{y}\tilde{u}(\cdot,y)\|_{L^2(\Omega)}
    \lesssim  \|\lim\limits_{y\rightarrow 0}(y^{-2s}  (|u|\ast K_{0,y}) + y^{2-2s} (|u|\ast K_{1,y}))\|_{L^2(\Omega)}
    \\ & = \left\|\lim\limits_{y\rightarrow 0}\left(\int_{\R^n\setminus B_{R+r}} \frac{|u(z)|}{(|x'-z|^2+y^2)^{n/2+s}}dz + \int_{\R^n\setminus B_{R+r}} \frac{|u(z)| y^{2}}{(|x'-z|^2+y^2)^{n/2+s+1}}dz\right)\right\|_{L^2(\Omega)}
    \\ & \lesssim
    \left\|\lim\limits_{y\rightarrow 0}\int_{\R^n\setminus B_{R+r}} \frac{|u(z)|}{(|x'-z|^2+y^2)^{n/2+s}}dz \right\|_{L^2(\Omega)}
    \\ & \leq    \|u\|_{L^q(\R^n)}\left\|\lim\limits_{y\rightarrow 0}\left(\int_{\R^n\setminus B_{R+r}} \frac{dz}{(|x'-z|^2+y^2)^{p(n/2+s)}}\right)^{1/p} \right\|_{L^2(\Omega)}
    \\ & \leq    \|u\|_{L^q(\R^n)}\left\|\left(\int_{\R^n\setminus B_{R+r}} \frac{dz}{|x'-z|^{p(n+2s)}}\right)^{1/p} \right\|_{L^2(\Omega)}
    \\ & \leq    \|u\|_{L^q(\R^n)}\left(\int_{\R^n\setminus B_{r}} \frac{dz}{|z|^{p(n+2s)}}\right)^{1/p},
\end{align*}
where $p:=\frac{2n}{n+2}$ is the conjugated exponent of $q$. Since the last integral is finite, the proof is complete.
\end{proof}

\subsection*{Acknowledgements}

Giovanni Covi was supported by an Alexander-von-Humboldt postdoctoral fellowship. The research of Tuhin Ghosh was supported by the Collaborative Research Center 1283, Universität Bielefeld. Angkana Rüland was supported by the Deutsche Forschungsgemeinschaft (DFG, German Research Foundation) through the Hausdorff Center for Mathematics under Germany’s Excellence Strategy - EXC-2047/1 - 390685813. Gunther Uhlmann was partly
supported by NSF and  a Robert R. and Elaine F. Phelps  Endowed
Professorship.

\bibliographystyle{alpha}
\bibliography{citations_v7}

\begin{thebibliography}{CMRU22}

\bibitem[Ale88]{A88}
Giovanni Alessandrini.
\newblock Stable determination of conductivity by boundary measurements.
\newblock {\em Appl. Anal.}, 27(1-3):153--172, 1988.

\bibitem[BGU21]{BGU21}
Sombuddha Bhattacharyya, Tuhin Ghosh, and Gunther Uhlmann.
\newblock Inverse problems for the fractional-laplacian with lower order
  non-local perturbations.
\newblock {\em Transactions of the American Mathematical Society},
  374(5):3053--3075, 2021.

\bibitem[CdHS22]{CDHS22}
Giovanni Covi, Maarten de~Hoop, and Mikko Salo.
\newblock Uniqueness in an inverse problem of fractional elasticity.
\newblock {\em arXiv preprint arXiv:2209.15316}, 2022.

\bibitem[CJKS20]{CJKS20}
T.~Coulhon, R.~Jiang, P.~Koskela, and A.~Sikora.
\newblock {Gradient estimates for heat kernels and harmonic functions}.
\newblock {\em Journal of Functional Analysis}, 278(8), 2020.

\bibitem[CK19]{CK}
David Colton and Rainer Kress.
\newblock {\em Inverse acoustic and electromagnetic scattering theory},
  volume~93 of {\em Applied Mathematical Sciences}.
\newblock Springer, Cham, [2019] \copyright 2019.
\newblock Fourth edition of [ MR1183732].

\bibitem[CLR20]{CLR20}
Mihajlo Ceki{\'c}, Yi-Hsuan Lin, and Angkana R{\"u}land.
\newblock The {C}alder{\'o}n problem for the fractional {S}chr{\"o}dinger
  equation with drift.
\newblock {\em Calculus of Variations and Partial Differential Equations},
  59(3):91, 2020.

\bibitem[CMRU22]{CMRU20}
Giovanni Covi, Keijo M\"{o}nkk\"{o}nen, Jesse Railo, and Gunther Uhlmann.
\newblock The higher order fractional {C}alder\'{o}n problem for linear local
  operators: uniqueness.
\newblock {\em Adv. Math.}, 399:Paper No. 108246, 29, 2022.

\bibitem[Cov20a]{C20b}
Giovanni Covi.
\newblock An inverse problem for the fractional {S}chr{\"o}dinger equation in a
  magnetic field.
\newblock {\em Inverse Problems}, 36(4):045004, 2020.

\bibitem[Cov20b]{C20}
Giovanni Covi.
\newblock Inverse problems for a fractional conductivity equation.
\newblock {\em Nonlinear Analysis}, 193:111418, 2020.

\bibitem[CR16]{CR16}
Pedro Caro and Keith~M Rogers.
\newblock {Global uniqueness for the Calder{\'o}n problem with Lipschitz
  conductivities}.
\newblock In {\em Forum of Mathematics, Pi}, volume~4, page~e2. Cambridge
  University Press, 2016.

\bibitem[CR21]{CR21}
Giovanni Covi and Angkana R{\"u}land.
\newblock On some partial data {C}alder{\'o}n type problems with mixed boundary
  conditions.
\newblock {\em Journal of Differential Equations}, 288:141--203, 2021.

\bibitem[CS07]{CS07}
Luis Caffarelli and Luis Silvestre.
\newblock An extension problem related to the fractional {L}aplacian.
\newblock {\em Communications in partial differential equations},
  32(8):1245--1260, 2007.

\bibitem[CS14]{CS14}
Xavier Cabr{\'e} and Yannick Sire.
\newblock Nonlinear equations for fractional {L}aplacians, {I}: {R}egularity,
  maximum principles, and {H}amiltonian estimates.
\newblock In {\em Annales de l'Institut Henri Poincare (C) Non Linear
  Analysis}, volume~31, pages 23--53. Elsevier, 2014.

\bibitem[Fei21]{F21}
Ali Feizmohammadi.
\newblock Fractional {C}alder{\'o}n problem on a closed {R}iemannian manifold.
\newblock {\em arXiv preprint arXiv:2110.07500}, 2021.

\bibitem[FF14]{FF14}
Mouhamed~Moustapha Fall and Veronica Felli.
\newblock Unique continuation property and local asymptotics of solutions to
  fractional elliptic equations.
\newblock {\em Communications in Partial Differential Equations},
  39(2):354--397, 2014.

\bibitem[FGKU21]{FGKU21}
Ali Feizmohammadi, Tuhin Ghosh, Katya Krupchyk, and Gunther Uhlmann.
\newblock Fractional anisotropic {C}alder{\'o}n problem on closed {R}iemannian
  manifolds.
\newblock {\em arXiv preprint arXiv:2112.03480}, 2021.

\bibitem[GFRZ22]{GFRZ22}
Mar{\'\i}a~{\'A}ngeles Garc{\'\i}a-Ferrero, Angkana R{\"u}land, and Wiktoria
  Zato{\'n}.
\newblock Runge approximation and stability improvement for a partial data
  {C}alder{\'o}n problem for the acoustic {H}elmholtz equation.
\newblock {\em Inverse Problems and Imaging}, 16(1):251--281, 2022.

\bibitem[Gri95]{Gri95}
A.~Grigor'yan.
\newblock {Upper bounds of derivatives of the heat kernel on an arbitrary
  complete manifold}.
\newblock {\em J. Funct. Anal.}, 127(2):363--389, 1995.

\bibitem[GRSU20]{GRSU20}
Tuhin Ghosh, Angkana R{\"u}land, Mikko Salo, and Gunther Uhlmann.
\newblock {Uniqueness and reconstruction for the fractional Calder{\'o}n
  problem with a single measurement}.
\newblock {\em Journal of Functional Analysis}, 279(1):108505, 2020.

\bibitem[GSU20]{GSU20}
Tuhin Ghosh, Mikko Salo, and Gunther Uhlmann.
\newblock {The Calder{\'o}n problem for the fractional {S}chr{\"o}dinger
  equation}.
\newblock {\em Analysis \& PDE}, 13(2):455--475, 2020.

\bibitem[GU21]{GU21}
Tuhin Ghosh and Gunther Uhlmann.
\newblock The {C}alder{\'o}n problem for nonlocal operators.
\newblock {\em arXiv preprint arXiv:2110.09265}, 2021.

\bibitem[Hab15]{H15}
Boaz Haberman.
\newblock {Uniqueness in Calder{\'o}n’s problem for conductivities with
  unbounded gradient}.
\newblock {\em Communications in Mathematical Physics}, 340:639--659, 2015.

\bibitem[HL19]{HL19}
Bastian Harrach and Yi-Hsuan Lin.
\newblock Monotonicity-based inversion of the fractional {S}chr{\"o}dinger
  equation {I}. {P}ositive potentials.
\newblock {\em SIAM Journal on Mathematical Analysis}, 51(4):3092--3111, 2019.

\bibitem[HL20]{HL20}
Bastian Harrach and Yi-Hsuan Lin.
\newblock Monotonicity-based inversion of the fractional {S}chr{\"o}dinger
  equation {II}. {G}eneral potentials and stability.
\newblock {\em SIAM Journal on Mathematical Analysis}, 52(1):402--436, 2020.

\bibitem[HT13]{HT13}
Boaz Haberman and Daniel Tataru.
\newblock {Uniqueness in Calder{\'{o}}n's problem with Lipschitz
  conductivities}.
\newblock {\em Duke Mathematical Journal}, 162(3), feb 2013.

\bibitem[Li20]{L20}
Li~Li.
\newblock {The Calder{\'o}n problem for the fractional magnetic operator}.
\newblock {\em Inverse Problems}, 36(7):075003, 2020.

\bibitem[Li21]{L21}
Li~Li.
\newblock Determining the magnetic potential in the fractional magnetic
  {C}alder{\'o}n problem.
\newblock {\em Communications in Partial Differential Equations},
  46(6):1017--1026, 2021.

\bibitem[LLR20]{LLR20}
Ru-Yu Lai, Yi-Hsuan Lin, and Angkana R{\"u}land.
\newblock The {C}alder{\'o}n problem for a space-time fractional parabolic
  equation.
\newblock {\em SIAM Journal on Mathematical Analysis}, 52(3):2655--2688, 2020.

\bibitem[LO20]{LO20}
Ru-Yu Lai and Laurel Ohm.
\newblock Inverse problems for the fractional {L}aplace equation with lower
  order nonlinear perturbations.
\newblock {\em arXiv preprint arXiv:2009.07883}, 2020.

\bibitem[LTU03]{LTU03}
Matti Lassas, Michael Taylor, and Gunther Uhlmann.
\newblock The {D}irichlet-to-{N}eumann map for complete {R}iemannian manifolds
  with boundary.
\newblock {\em Communications in Analysis and Geometry}, 11(2):207--221, 2003.

\bibitem[LU89]{LU89}
John~M Lee and Gunther Uhlmann.
\newblock Determining anisotropic real-analytic conductivities by boundary
  measurements.
\newblock {\em Communications on Pure and Applied Mathematics},
  42(8):1097--1112, 1989.

\bibitem[LU01]{LU01}
Matti Lassas and Gunther Uhlmann.
\newblock On determining a {R}iemannian manifold from the
  {D}irichlet-to-{N}eumann map.
\newblock In {\em Annales scientifiques de l'Ecole normale sup{\'e}rieure},
  volume~34, pages 771--787, 2001.

\bibitem[McL00]{McLean}
William McLean.
\newblock {\em Strongly elliptic systems and boundary integral equations}.
\newblock Cambridge University Press, Cambridge, 2000.

\bibitem[Nac88]{N88}
Adrian~I Nachman.
\newblock Reconstructions from boundary measurements.
\newblock {\em Annals of Mathematics}, 128(3):531--576, 1988.

\bibitem[RS18]{RS18}
Angkana R{\"u}land and Mikko Salo.
\newblock Exponential instability in the fractional {C}alder{\'o}n problem.
\newblock {\em Inverse Problems}, 34(4):045003, 2018.

\bibitem[RS19a]{RS19}
Angkana R{\"u}land and Mikko Salo.
\newblock Quantitative {R}unge approximation and inverse problems.
\newblock {\em International Mathematics Research Notices},
  2019(20):6216--6234, 2019.

\bibitem[RS19b]{RS19a}
Angkana R{\"u}land and Eva Sincich.
\newblock {Lipschitz stability for the finite dimensional fractional
  {C}alder{\'o}n problem with finite Cauchy data}.
\newblock {\em Inverse Problems \& Imaging}, 13(5), 2019.

\bibitem[RS20]{RS20}
Angkana R{\"u}land and Mikko Salo.
\newblock The fractional {C}alder{\'o}n problem: low regularity and stability.
\newblock {\em Nonlinear Analysis}, 193:111529, 2020.

\bibitem[R{\"u}l15]{R15}
Angkana R{\"u}land.
\newblock {Unique continuation for fractional Schr{\"o}dinger equations with
  rough potentials}.
\newblock {\em Communications in Partial Differential Equations},
  40(1):77--114, 2015.

\bibitem[R{\"u}l18]{R18a}
Angkana R{\"u}land.
\newblock {Unique continuation, Runge approximation and the fractional
  Calder{\'o}n problem}.
\newblock {\em Journ{\'e}es {\'e}quations aux d{\'e}riv{\'e}es partielles},
  pages 1--10, 2018.

\bibitem[R{\"u}l19]{R19}
Angkana R{\"u}land.
\newblock Quantitative invertibility and approximation for the truncated
  {H}ilbert and {R}iesz transforms.
\newblock {\em Revista Matem{\'a}tica Iberoamericana}, 35(7):1997--2024, 2019.

\bibitem[R{\"u}l21]{R21}
Angkana R{\"u}land.
\newblock On single measurement stability for the fractional {C}alder{\'o}n
  problem.
\newblock {\em SIAM Journal on Mathematical Analysis}, 53(5):5094--5113, 2021.

\bibitem[Sal17]{S17}
Mikko Salo.
\newblock {The fractional Calder{\'o}n problem}.
\newblock {\em Journ{\'e}es {\'e}quations aux d{\'e}riv{\'e}es partielles},
  pages 1--8, 2017.

\bibitem[ST10]{ST10}
Pablo~Ra{\'u}l Stinga and Jos{\'e}~Luis Torrea.
\newblock Extension problem and {H}arnack's inequality for some fractional
  operators.
\newblock {\em Communications in Partial Differential Equations},
  35(11):2092--2122, 2010.

\bibitem[SU87]{SU87}
John Sylvester and Gunther Uhlmann.
\newblock A global uniqueness theorem for an inverse boundary value problem.
\newblock {\em Ann. of Math.}, 125(1):153--169, 1987.

\bibitem[Uhl09]{U09}
Gunther Uhlmann.
\newblock Electrical impedance tomography and {C}alder{\'o}n's problem.
\newblock {\em Inverse Problems}, 25(12):123011, 2009.

\end{thebibliography}

\end{document}